\documentclass[12pt]{article}
\usepackage{bbm, amsmath, amsfonts, amscd, latexsym, amsthm, amssymb, graphicx, euscript, mathrsfs,refcount,hyperref}
\usepackage{mathtools}
\usepackage[nottoc,notlot,notlof]{tocbibind}
\usepackage[dvipsnames]{xcolor}
\usepackage{tikz}
\usetikzlibrary{patterns}

\newtheorem{theor}{\hspace{1cm}{\sc Theorem}}[section]
\newtheorem{utver}[theor]{\hspace{1cm}{\sc Proposition}}
\newtheorem{predpol}[theor]{\hspace{1cm}{\sc Assumption}}
\newtheorem{sledst}[theor]{\hspace{1cm}{\sc Corollary}}
\newtheorem{lemma}[theor]{\hspace{1cm}{\sc Lemma}}
\newtheorem{conj}[theor]{\hspace{1cm}{\sc Conjecture}}
\newtheorem*{utver*}{\hspace{1cm}{\sc Proposition}}
\theoremstyle{definition}
\newtheorem{defin}[theor]{\hspace{1cm}{\sc Definition}}
\newtheorem{exa}[theor]{\hspace{1cm}{\sc Example}}
\newtheorem{rem}[theor]{\hspace{1cm}{\sc Remark }}
\newtheorem{prb}[theor]{\hspace{1cm}{\sc Problem}}

\newcommand{\ind}{\mathop{\rm ind}\nolimits}

\newcommand{\codim}{\mathop{\rm codim}\nolimits}

\newcommand{\vol}{\mathop{\rm Vol}\nolimits}
\newcommand{\area}{\mathop{\rm Area}\nolimits}
\newcommand{\conv}{\mathop{\rm conv}\nolimits}

\newcommand{\MV}{\mathop{\rm MV}\nolimits}

\newcommand{\MP}{\mathop{\rm MP}\nolimits}

\newcommand{\supp}{\mathop{\rm supp}\nolimits}

\newcommand{\bigslant}[2]{{\raisebox{.2em}{$#1$}\left/\raisebox{-.2em}{$#2$}\right.}}
\def\R{\mathbb R}
\def\N{\mathbb N}
\def\Z{\mathbb Z}

\def\C{\mathbb C}
\def\CC{({\mathbb C}\setminus 0)}

\def\newton{\EuScript{N}}

\begin{document}
\title{On the Singular Locus of a Plane Projection of a Complete Intersection}
\author{Arina Voorhaar\thanks{{\it National Research University Higher School of Economics}, Moscow, Russia}\thanks{{\it University of Geneva}, Geneva, Switzerland}}
\date{}
\maketitle{}
\begin{abstract}
In this paper, we compute the number of self-intersections of a plane projection of a generic complete intersection curve defined by polynomials with the given support. Moreover, we discuss the tropical counterpart of this problem. 
\end{abstract}
\tableofcontents
\section{Introduction}
One of the main tools in the study of singularities of maps is the theory of Thom polynomials. They express the fundamental classes of the multisingularity strata for a generic map of arbitrary compact smooth manifolds in terms of their characteristic classes. This theory is however not applicable to a natural class of generic maps, namely the maps between varieties that are defined by generic Laurent polynomials with given Newton polytopes. Such varieties are not compact, and the maps are not proper. At the same time, their toric compactifications associated with the Newton polytopes do not satisfy the genericity conditions necessary for classical Thom polynomials to be applicable (for details see e.g. Example 1.1 and Remark 1.2 in \cite{E3}). Therefore, working with multisingularity strata of the abovementioned class of maps requires alternative methods. We will now give a short overview of developments in this direction.

The discriminant (i.e. $\mathcal A_1$ stratum) of a projection of a generic hypersurface was described in \cite{GP}. If $H$ is a hypersurface given by a generic polynomial $f(x_1,\ldots,x_d,y)$ and $\pi$ is the projection forgetting the last coordinate, then the Newton polytope of the polynomial defining the abovementioned  $\mathcal A_1$ stratum is equal to the fiber polytope $\mathcal Q_{\pi}(f)\subset\R^d$ of the Newton polytope $\newton(f).$ 

The image (i.e. $\mathcal A_0$ stratum) of a projection of a generic complete intersection was studied by A. Esterov and A. Khovanskii in \cite{EK}. They proved that the image under an epimorphism  $\pi\colon\CC^{n}\to\CC^{n-k}$ of a complete intersection $\{f_1=\ldots=f_{k+1}\}\subset\CC^n$ defined by generic polynomials with given Newton polytopes $\newton(f_i)=\Delta_i$ is a hypersurface $\{g=0\}\subset\CC^{n-k},$ whose Newton polytope $\newton(g)\subset\R^{n-k}$ is equal to the so-called {\it mixed fiber polytope} $\MP_{\pi}(\Delta_1,\ldots,\Delta_{k+1})$ of the polytopes $\Delta_1,\ldots,\Delta_{k+1}.$ 

An approach to studying the strata of higher codimension, e.g for $\mathcal A_2$ (cusps) and $2\mathcal A_1$ (double points), is suggested in \cite{E3}. However, this approach only works under additional assumptions on the Newton polytopes and employs operations with tropical fans of dimension that is too high for practical applications (namely, those dimensions are of the same order as the number of monomials inside the given Newton polytopes). 

For low dimensional maps or projections, the above mentioned strata are $0$-dimensional, and the problem of computing their cardinalities in terms of the given Newton polytopes arises naturally. 
For example, this problem was solved in \cite{FJR} for a mapping $\C^2\to\C^2$ whose components are generic polynomials of given degrees. See \cite{E3} for an overview of some other literature on problems of this kind.

To the best of our knowledge, our paper is the first work where such a problem is solved for polynomials with arbitrary Newton polytopes. 

\begin{theor}\label{mainintro}
Let $\Delta$ be a lattice polytope in $\Z^n\oplus \Z^2.$ If, for generic polynomials $f_1,\ldots,f_{n+1}$ supported at $\Delta,$ the image of the complete intersection curve $$\tilde{\mathcal C}={\{f_1=\ldots=f_{n+1}=0\}\subset\CC^n\times\CC^2}$$ under the projection $\pi\colon\CC^n\times\CC^2\to\CC^2$ is a reduced nodal curve, then the number of its nodes is given by formula (\ref{mainformula1}) in Theorem \ref{lemmain}. 
\end{theor}

\begin{rem}
The classification of the Newton polytopes, such that the abovementioned projection is not a nodal curve, is a non-trivial problem (see Example \ref{notnodal}) that is not addressed in this paper. Instead, we will prove a certain generalization of Theorem \ref{mainintro} which is applicable to all lattice polytopes $\Delta$ (see Theorem \ref{lemmain}).
\end{rem}

\begin{exa}[Counting the nodes of the projection of a complete intersection curve defined by generic equations of given degree]\label{simplex}
	For some support sets $A\subset\Z^{n+2},$ it is quite easy to show that the projection of a complete intersection given by generic polynomials supported at $A$ has only nodes as singularities. For instance, this is the case for $A=dT\cap\Z^3,$ where ${d\in\Z_{>0},}$ and $T\subset\R^3$ is the standard simplex. Let us compute the number of those nodes using (\ref{mainformula1}).  In the notation of this formula, we have  $n=1,$ the area of the polygon $P$ is equal to $d^4,$ the term $d^2$ comes from the area of the horizontal facet of $dT$, and the non-horizontal facets do not contribute, since for every $\Gamma\in\mathcal{F}(dT)\setminus\mathcal{H}(dT),$ we have $\ind_{v}(\Gamma\cap A)=1.$
	Thus, the answer is $$\mathcal{D}=\dfrac{d^4-2d ^3+d^2}{2}=\dfrac{d^2(d-1)^2}{2}.$$
\end{exa}      

This problem might also be of interest with regard to the study of algebraic knots, which is motivated by Viro's work \cite{V} about the rigid isotopy invariant called encomplexed writhe, as well as the works of Mikhalkin and Orevkov (see e.g. \cite{MO1}, \cite{MO2}
, \cite{MO3}). Namely, it is quite natural to estimate the complexity of an algebraic knot/link (that is, the minimal crossing number) in terms of the algebraic complexity of its defining equations (that is, the size of their Newton polytopes).

Theorem \ref{mainintro} gives an upper bound for the number of self-intersections of a projection of a real complete intersection curve onto a coordinate plane. 
In Example \ref{real} below we show that this upper bound is sharp for the case of real complete intersection links given by a pair of polynomials of given degree. 
Much harder problems, such as obtaining a sharp upper bound for arbitrary plane projections (not only onto coordinate planes), or topological classification of complete intersection links given by polynomials with arbitrary Newton polytopes, still remain unsolved.

\begin{exa}[real version of Example \ref{simplex}]\label{real}
Take two generic unions of $d$ planes as the hypersurfaces. Thus, we obtain a pair of varieties given by polynomials which are just products of the corresponding linear factors. Then, their intersection consists of $d^2$ lines, and the projection of this union of lines has $\dfrac{d^2(d^2-1)}{2}$ double points. Those double points are of two types: the ones which are stable and the ones which are not. The latter will disappear when we perturb the hypersurfaces. So, to find the sought number of stable double points, we need to compute the number of the double points that will disappear. Those are exactly the double points, whose preimages are intersections of a plane contained in one of the unions with the line of intersection of two planes contained in the other union. So, from $\dfrac{d^2(d^2-1)}{2}$, which is the total number of double points, we need to subtract the sum $\dfrac{d^2(d-1)}{2}+\dfrac{d^2(d-1)}{2}=d^2(d-1)$, and what we obtain as the result is exactly $\dfrac{d^2(d-1)^2}{2}$ nodes. Therefore, the upper bound obtained in the complex case is sharp for the real case. 
\end{exa}

\noindent {\bf Methodology and structure of the paper.} 
The Newton polytopes of objects such as the image or the discriminant $M$ of the projection of a complete intersection are known (see \cite{E2}, \cite{EK}, \cite{GP}). 
Therefore a natural first step in the study of the singularities of $M$ would be passing to the corresponding toric compactification $\bar{M}.$ 
If $\bar{M}$ did not have any additional singularities, then the problem of describing the simplest singularity strata of $M$, such as $\mathcal A_2$ and $2\mathcal A_1$, 
could be solved using classical methods. Unfortunately, the compactification $\bar{M}$ does in general have singularities at infinity. Moreover, these 
singularities are significantly more complicated than the ones studied. Dealing with them turns out to be the most challenging part in this class of problems. 
We reduce it to the so-called {\it forking--path singularities}.

This paper is organized as follows. Section \ref{preliminaries} is devoted to the notions and results that will be used throughout the paper. In Section \ref{mainresult} we fix the notation and state the main result of the paper: Theorem \ref{lemmain}. Section \ref{mainproof} is devoted to the proof of the main result. Finally, in Section \ref{fd} we discuss a few questions that naturally arise in the context of this article, including the tropical counterpart of the main problem. 

\noindent {\bf Acknowledgements.} This paper and the research behind it would not have been possible without the exceptional support, guidance and constant feedback of Alexander Esterov. I am also very grateful to Andr\'{a}s Szenes for his encouragement, help, and stimulating conversations. I would like to thank Patrick Popescu-Pampu and Grigory Mikhalkin for fruitful discussions, valuable comments and suggestions. The research was partially supported by the NCCR SwissMAP of the Swiss National Science Foundation.\\

\section{Preliminaries}\label{preliminaries}
\subsection{Newton Polytopes and Bernstein--Kouchnirenko Theorem}
\begin{defin}
	Let $f(x)=\sum_{a\in \Z^n} c_a x^a$ be a Laurent polynomial. The {\it support of $f$} is the set $\supp(f)\subset \Z^n$ which consists of all the points $a\in \Z^n$ such that the corresponding coefficient $c_a$ of $f$ is non-zero. 
\end{defin}

\begin{defin}
	The {\it Newton polytope of $f(x)$} is the convex hull of $\supp(f)$ in $\R^n$. In other words, it is the minimal convex lattice polytope in $\R^n$ containing the set $\supp(f)$. We denote the Newton polytope of a polynomial $f(x)$ by $\newton(f)$. 
\end{defin}

\begin{defin} \label{minkowski}
For a pair of subsets $A, B\subset\R^n$, their {\it Minkowski sum} is defined to be the set $A+B=\{a+b\mid a\in A, b\in B\}$. 
\end{defin}
The following fact provides a connection between the operations of the Minkowski addition and multiplication in the ring of Laurent polynomials.
\begin{utver} 
Let $f,g$ be a pair of Laurent polynomials. Then we have the following equality: $$\newton(fg)=\newton(f)+\newton(g).$$
\end{utver}

\begin{defin}\label{suppface}
	Let $A\subset\R^m$ be a convex lattice polytope and $\ell\in(\R^m)^*$ be a covector. Consider $\ell$ as a linear function, and denote by $\ell\mid_A$ its restriction to the polytope $A$. The function $\ell\mid_A$ attains its maximum at some face $\Gamma\subset A$. This face is called {\it the support face} of the covector $\ell$ and is denoted by $A^{\ell}$.
\end{defin}

\begin{defin}
	Let $\gamma\neq 0$ in $(\R^n)^*$ be a covector and $f(x)$ be a Laurent polynomial with the Newton polytope $\newton(f)$. The {\it truncation} of $f(x)$ with respect to $\gamma$ is the polynomial $f^{\gamma}(x)$ that is obtained from $f(x)$ by omitting the sum of monomials which are not contained in the support face ${\newton(f)^{\gamma}}$. 
\end{defin}

It is easy to show that for a system of equations $\{f_1(x)=\ldots=f_n(x)=0\}$ and an arbitrary covector $\gamma\neq 0$, the  system $\{f_1^{\gamma}(x)=\ldots=f_n^{\gamma}(x)=0\}$ by a monomial change of variables can be reduced to a system in $n-1$ variables at most. Therefore, for the systems with coefficients in general position, the ``truncated" systems are inconsistent in $\CC^n$. 

\begin{defin}
	Let $(f_1,\ldots, f_n)$ be a tuple of Laurent polynomials. In the same notation as above, the system ${f_1(x)=\ldots=f_n(x)=0}$ is called {\it Newton-nondegenerate}, if for any covector $0\neq\gamma\in(\R^n)^*$ the system $\{f_1^{\gamma}(x)=\ldots=f_n^{\gamma}(x)=0\}$ is not consistent in $\CC^n$.
\end{defin}

\begin{defin} \label{mv}
	Let $\mathscr P$ be the semigroup of all convex polytopes in $\R^n$ with respect to the Minkowski addition (see Definition \ref{minkowski}). The {\it mixed volume} is a unique function $$\MV\colon\underbrace{\mathscr{P}\times\ldots\times\mathscr{P}}_{\mbox{$n$ times}}\to\R$$ which symmetric, multilinear (with respect to the Minkowski addition) and which satisfies the following property: the equality $MV(P,\ldots, P)=\vol(P)$ holds for every polytope~${P\in \mathscr P}$.
\end{defin}

The following theorem allows to compute the number of roots for a non-degenerate polynomial system of equations in terms of the mixed volume of its Newton polytopes. 

\begin{theor}[Bernstein--Kouchnirenko formula, \cite{B}]
	The number of roots for a Newton-nondegenerate system of polynomial equations $\{f_1(x)=\ldots=f_n(x)=0\}$ in $\CC^n$ counted with multiplicities is equal to $n!\MV(\newton(f_1),\ldots, \newton(f_n))$.
\end{theor} 

\subsection{Fibers of Resultantal Sets}
This Subsection is devoted to one of the key results that we will use in this paper -- Theorem \ref{num_fib}. We refer the reader to the work \cite{E1} for proofs and details. Before we introduce all the necessary notation and state the theorem itself, let us consider the following example. 

\begin{exa}\label{1_dim_exa}
Let $\{f_1(x)=f_2(x)=0\}$ be a system of univariate equations with $\supp(f_1)=\{0,2\}\subset\Z$ and $\supp(f_2)=\{1,3\}\subset\Z.$ Thus we have  $f_1(x)=a+bx^2$ and  $f_2(x)=cx+dx^3.$ It is a well-known fact that such a system of equations is consistent if and only if its coefficients satisfy the equality $a(ad-bc)=0.$ 
Now, suppose that we take sufficiently generic polynomials $f_1(x)$ and $f_2(x)$ whose coefficients satisfy this relation. A very natural question is the following: how many solutions does such a system have? In this particular case it is quite easy to find the answer by hand: the system $\{f_1(x)=f_2(x)=0\}$ has $2$ roots. 

As we will see further, this answer can be obtained purely in terms of the combinatorial properties of the support sets $\supp(f_1)$ and $\supp(f_2).$ Namely, note that these sets can be shifted to the same proper sublattice $2\Z\subset \Z$ of index $2.$  
\end{exa}

Let $\mathcal A=(A_1,\ldots,A_{I})$ be a collection of finite sets $A_i\in\Z^n.$ 

\begin{defin}
If $I\neq 0,$ then the number $I-\dim\sum\limits_{1\leqslant i\leqslant I} A_i,$ where sumation is in the sense of Minkowski, is called {\it the codimension} of the collection $\mathcal A.$  The codimension of the empty collection is set to be $0.$
\end{defin}

\begin{exa}
The codimension of the collection $\mathcal A=(\{0,2\},\{1,3\})$ considered in Example \ref{1_dim_exa} is equal to $1.$ 
\end{exa}

\begin{defin}
A collection $\mathcal A=(A_1,\ldots,A_{I})$ of finite sets in $\Z^n$ is called {\it essential,} if its codimension is strictly greater than the codimension of every subcollection $A_{i_1} ,\ldots,A_{i_k},$ where $\{i_1,...,i_k\}$ is a subset of $\{1,...,I\}.$	
\end{defin}

\begin{exa}\label{ess}
The collection $\mathcal A=(\{0,2\},\{1,3\})$ considered in Example \ref{1_dim_exa} is essential, since the codimension of $\mathcal A$ is equal to $1$, while the codimensions of all of its subcollections are equal to $0.$
\end{exa}

\begin{defin}
A collection $\mathcal A=(A_1,\ldots,A_{I})$ of finite sets in $\Z^n$ is called {\it weakly essential,} if its codimension is not smaller than the codimension of each of the  subcollections $(A_{i_1} ,\ldots,A_{i_k})$ where $\{i_1,\ldots,i_k\}\subset
\{1,\ldots,I\}.$	
\end{defin}

\begin{exa}\label{w_ess}
The collection $\mathcal A=(A_1,A_2,A_3,A_4)$ of finite subsets in $\Z^3$ shown in  Figure $1$ below is weakly essential. 

\begin{center}
\begin{tikzpicture}
\draw[thick] (0,0)--(0,2);
\draw[thick] (2,0)--(2,2)--(4,0)--(2,0);
\draw[thick] (6,0)--(6,2)--(8,0)--(6,0);
\draw[thick] (10,0)--(10,2)--(11,-1)--(10,0);
\draw[thick] (11,-1)--(12,0)--(10,2)--(11,-1);
\draw[thick, dashed] (10,0)--(12,0);
\draw[fill,red] (0,0) circle [radius=0.1];
\draw[fill,red] (0,2) circle [radius=0.1];
\draw[fill,red] (2,0) circle [radius=0.1];
\draw[fill,red] (2,2) circle [radius=0.1];
\draw[fill,red] (4,0) circle [radius=0.1];
\draw[fill,red] (6,0) circle [radius=0.1];
\draw[fill,red] (6,2) circle [radius=0.1];
\draw[fill,red] (8,0) circle [radius=0.1];
\draw[fill,red] (10,0) circle [radius=0.1];
\draw[fill,red] (12,0) circle [radius=0.1];
\draw[fill,red] (10,2) circle [radius=0.1];
\draw[fill,red] (11,-1) circle [radius=0.1];
\node[left] at (0,0) {$(0,0,0)$};
\node[below] at (2,0) {$(0,0,0)$};
\node[below] at (6,0) {$(0,0,0)$};
\node[above left] at (10,0) {$(0,0,0)$};
\node[left] at (0,2) {$(1,0,0)$};
\node[right] at (2,2) {$(1,0,0)$};
\node[right] at (10,2) {$(1,0,0)$};
\node[left] at (0,2) {$(1,0,0)$};
\node[right] at (6,2) {$(1,0,0)$};
\node[below] at (4,0) {$(0,1,0)$};
\node[below] at (8,0) {$(0,1,0)$};
\node[right] at (11,-1) {$(0,0,1)$};
\node[right] at (12,0) {$(0,1,0)$};
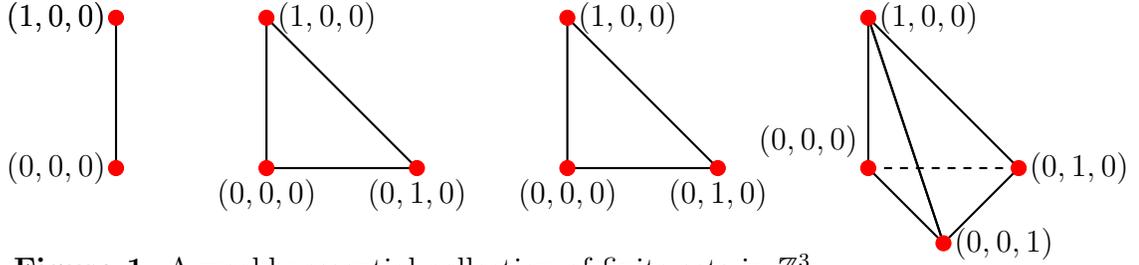
\node[below] at (4,-1) {{\bf Figure 1.} A weakly essential collection of finite sets in $\Z^3$.};
\end{tikzpicture}
\end{center}
Indeed, its codimension is equal to $1,$ which is, for instance, equal to the codimension of the subcollection $(A_1,A_2,A_3).$
\end{exa}

\begin{lemma}
Every collection $\mathcal A=(A_1,\ldots,A_I )\subset \Z^n$ contains a unique essential subcollection of codimension not less than the codimensions of all other subcollections of $\mathcal A.$ 
\end{lemma}

\begin{exa}
The subcollection $(A_1,A_2,A_3)$ of the collection $\mathcal A=(A_1,A_2,A_3,A_4)$ in Example \ref{w_ess} is the unique essential subcollection of the same codimension as $A.$ 
\end{exa}

\begin{defin}\label{quot_coll}
Let $\mathcal A=(A_1,\ldots,A_I)$ be a collection of finite subsets of $\Z^n,$ and let $(A_{i_1},\ldots,A_{i_k})$ be oe of its subcollections. Suppose that the sum $\sum\nolimits_{j=1}^kA_{i_j}$ generates a sublattice $L \subset\Z^n.$ Consider  $L'\subset\Z^n,$ the maximal sublattice of the same dimension as $L$ such that $L\subset L'.$  Denote the projection $\Z^n\twoheadrightarrow\bigslant{\Z^n}{L'}$ 
by $\rho.$ Then the collection of all sets of the
form $\rho(A_i), i \notin \{i_1,\ldots,i_k\}.$ is called the {\it quotient collection.}
\end{defin}

\begin{exa}\label{quot_exa}
Let us come back to the collection $\mathcal A=(A_1,A_2,A_3,A_4)$ of Example \ref{w_ess}. 
Take its the subcollection $(A_1,A_2,A_3).$ In this case, $L=L'=\langle (1,0,0),(0,1,0)\rangle,$ and the quotient collection $(B_1)$ that consists of the set $B_1=\rho(A_2)=\{0,1\}$ in $\bigslant{\Z^3}{L}\simeq\Z$. 
\end{exa}

\begin{defin}
In the notation of Definition \ref{quot_coll}, let $A_{i_1},\ldots,A_{i_k}$ be the essential subcollection of a collection of finite sets $\mathcal A=(A_1,\ldots,A_{I}) \subset \Z^n$  that
has the same codimension as $\mathcal A$, and let $B_1,\ldots,B_{I-k}$ be the quotient collection. The {\it multiplicity} of the collection $\mathcal A$ is denoted by $d(\mathcal A)$ and defined by the following formula: 
$$d(\mathcal A)=|L'/L|(I - k)! \MV(\conv B_1,\ldots, \conv B_{I-k}),$$
where $|L'/L|$ stands for the index of the sublattice $L\subset L'.$ 
\end{defin}

\begin{rem}
If the collection $\mathcal A$ is essential, then the quotient collection is empty. In this case, we follow the convention $\MV(\varnothing)=1.$
\end{rem}

\begin{exa}
Let us come back to Example \ref{quot_exa}. The multiplicity $d(\mathcal A)$ of the collection $\mathcal A$ is equal to $1!~\mathrm{Length}(\conv B_1)=1.$
\end{exa}

\begin{defin}
For a finite set $A\subset \Z^n, $ by $\C^A$ we denote the set of all Laurent
polynomials of the form $\sum\limits_{a\in A} c_at^a$, where $a=(a_1,\ldots,a_n), t=(t_1,\ldots,t_n)$ and $t^a=t_1^{a_1}\cdot\ldots\cdot t_n^{a_n}.$ 
Let $A_1,\ldots,A_I \subset \Z^n$ be finite sets. We define $\Sigma(A_1,...,A_I)$ to be the closure of the set of all collections
$(p_1,\ldots,p_I ) \in \C^{A_1}\oplus\ldots\oplus\C^{A_I}$ such that the set 
$$\{t\in \CC^n \mid p_1(t) = \ldots = p_I (t)=0\} \subset\CC^n$$ is not empty.
\end{defin}

\begin{theor}
Let $\mathcal A=(A_1,...,A_I)$ be a collection of finite sets in $\Z^n.$ The set $\Sigma(A_1,...,A_I)$ is irreducible, and its codimension is equal to the maximum over  the codimensions of all the subcollections of the collection $\mathcal A.$ 
\end{theor}

\begin{exa}\label{codim_s}
For the sets $A_1,A_2$ in Example \ref{ess}, we have $\codim(\Sigma(A_1,A_2))=1.$

The codimension of the set $\Sigma(A_1,A_2,A_3,A_4),$ where $A_1,A_2,A_3,A_4$ are the same as in Example \ref{w_ess}, is also equal to $1.$
\end{exa}

\begin{defin}
For a point $f = (f_1,...,f_I ) \in \Sigma(A_1,...,A_I),$ the $f$-fiber is defined to be the algebraic set $$\{t\in\CC^n \mid f_1(t) = \ldots = f_I (t)=0\}\subset\CC^n.$$	
\end{defin} 

The dimension of the $f$-fiber for a generic $f\in \Sigma(A_1,...,A_I)$ can be computed via the following formula.

\begin{lemma}
There exists a Zariski open subset $S\subset\Sigma(A_1,...,A_I)$ such that the dimension of the $f$-fiber is equal $n -I + \codim( \Sigma(A_1,...,A_I))$ for every $f\in S.$	
\end{lemma}

\begin{exa}
For both of the collections of finite sets considered in Example \ref{codim_s}, the dimension of the $f$-fiber is equal to $0,$ for generic $f.$ At the same time, if instead of any of these two collections we consider their images under an inclusion into some  lattice $\Z^k$ of higher dimension, then the dimension of the corresponding $f$-fiber will be non-zero. 
\end{exa}

Now we are ready to state the main result of this Subsection. 

\begin{theor}\label{num_fib}
	Suppose that a collection $\mathcal A=(A_1,\ldots,A_I)$ of finite sets in $\Z^n$ is weakly essential. Then there exists a Zariski open subset $S\subset \Sigma(A_1,\ldots,A_I )$ such that, for every $f\in S,$ the $f$-fiber is a disjoint
	union of $d(\mathcal A )$ shifted copies of a subtorus of the complex torus $\CC^n.$ 
\end{theor}

\begin{exa}
Applying Theorem \ref{num_fib} to the collection $\mathcal A=(A_1,A_2,A_3,A_4)$ considered in Example \ref{w_ess}, we obtain the following statement: there exists a Zariski open subset $S\subset\Sigma(A_1,A_2,A_3,A_4)$ such that, for every $f=(f_1,f_2,f_3,f_4)\in S,$ the system $\{f_1=f_2=f_3=f_4=0\}$ has exaclty one root. 
\end{exa}

\begin{exa}
For the essential collection considered in Example \ref{1_dim_exa}, we have that there exists a Zariski open subset $S\subset\Sigma(A_1,A_2)$ such that for every $f=(f_1,f_2)\in S,$ the system $\{f_1=f_2=0\}$ has exaclty $2$ roots. 
\end{exa}

\subsection{Fiber Polytopes}
This Subsection contains some basic facts about fiber polytopes, that we will use throughout the paper. For more details we refer the reader to the work \cite{EK}. 
\begin{defin}
Let $\pi\colon\CC^n\to\CC^{n-k}$ be an epimorphism of complex tori, and let  $f_1,\ldots,f_{k+1}$ be Laurent polynomials on $\CC^n$ defining a complete intersection $\tilde{X}=\{f_1=\ldots=f_{k+1}=0\}$ of codimension $k+1.$ Then the Laurent polynomial $\pi_{f_1,\ldots,f_{k+1}}$ defining the image $X\subset\CC^{n-k}$ of $\tilde{X}$ under the epimorphism $\pi$ is called {\it the composite polynomial} of polynomials $f_1,\ldots,f_{k+1}$ with respect to $\pi.$ 
\end{defin}

The composite polynomial is defined uniquely up to a monomial factor. Its Newton polytope is uniquely determined up to a shift according to Theorem \ref{comp} below. The goal of this Subsection is to describe the Newton polytope of the composite polynomial in the case when $\newton(f_1)=\ldots=\newton(f_{k+1})=\Delta.$

Let $\pi^{\times}\colon\Z^{n-k}\hookrightarrow\Z^n$ be the inclusion of character lattices induced by the epimorphism $\pi\colon\CC^n\to\CC^{n-k}.$ We denote $\Delta_j=\newton(f_j)$ for $1\leqslant j\leqslant k+1,$ and ${B=\newton(\pi_{f_1,\ldots,f_{k+1}}).}$ 

\begin{theor}[\cite{EK}]
In the same notation as above, if the polynomials $f_1,\ldots,f_{k+1}$ are Newton-nondegenerate, then for any convex bodies $B_1,\ldots,B_{n-k-1}\subset\Z^{n-k},$ the following equation holds: 
\begin{equation}
(n-k)!\MV(B,B_1,\ldots,B_{n-k-1})=n!\MV(\Delta_1,\ldots,\Delta_{k+1},\pi^{\times}B_1,\ldots,\pi^{\times}B_{n-k-1}).
\end{equation}
\end{theor}

\begin{defin}
Let $L\subset\R^n$ be a codimension $k$ vector subspace, and let $\mu$ and $\mu'$ be volume forms on $\bigslant{\R^n}{L}$ and $L$ respectively. Consider $\Delta_1,\ldots,\Delta_{k+1},$ a tuple of convex bodies in $\R^n.$ A convex body $B\subset L$ is called a {\it composite body} of $\Delta_1,\ldots,\Delta_{k+1}$ in $L,$ if, for every collection $B_1,\ldots,B_{n-k-1}\subset L,$ the following equality holds: 
\begin{equation}\label{vol_composite}
(n-k)!\MV_{\mu'}(B,B_1,\ldots,B_{n-k-1})=n!\MV_{\mu\wedge\mu'}(\Delta_1,\ldots,\Delta_{k+1},B_1,\ldots,B_{n-k-1}).
\end{equation}
\end{defin}

\begin{theor}\label{comp}
In the same notation as above, for any collection of convex bodies $\Delta_1,\ldots,\Delta_{k+1}\subset\R^n$ its composite body exists. Moreover, it is unique up to a shift. 
\end{theor}

We will now give an explicit description of the composite body for the special case ${\Delta_1=\ldots=\Delta_{k+1}=\Delta.}$

\begin{defin}\label{Mint_def}
Let $L\subset\R^n$ be a codimension $k$ vector subspace, and let $\mu$ be a volume form on $\bigslant{\R^n}{L}$. Denote by $p$ the projection $\R^n\twoheadrightarrow \bigslant{\R^n}{L}.$For a convex body $\Delta\subset\R^n,$ the set of all points of the form $\int_{p(\Delta)}s\mu\in\R^n,$ where $s\colon p(\Delta)\to\Delta$ is a continuous section of the projection $p,$ is called the {\it Minkowski integral}, or the {\it fiber body} of $\Delta$ and is denoted by $\int p|_{\Delta}\mu.$ 
\end{defin}

\begin{exa}\label{cylind}
By $p$ we denote the projection $p\colon\R^3\twoheadrightarrow\bigslant{\R^3}{L}.$ Consider $\Delta=Q\times I,$ with $Q\subset L=\R^2$ convex and $I\subset\ker p$ a closed interval. Using Definition \ref{Mint_def}, we immediately obtain $\int p|_{\Delta}\mu=|I|\cdot Q.$ 
\end{exa}

\begin{rem}
Definition \ref{Mint_def} can be visualized as follows. Let $p\colon\R^3\to\R^2$ be the projection forgetting the last coordinate. The Minkowski integral of the convex body $\Delta$ is defined as follows: $\int p|_{\Delta}\mu=\int\Delta(t) dt,$ where $\Delta(t)$ stands for the function which maps every $s\in\R$ to $\Delta\cap\{t=s\}\subset\R^2$ (see Figure 2 below). It is similar to the usual Riemann integral, but in this case, the function takes values in convex sets and the addition operation is the Minkowski sum. The body $\Delta$ can be approximated by cylinders, whose Minkowski integrals were computed in Example \ref{cylind}. So, the integral $\int p|_{\Delta}\mu$ is then the limit of the sums $\sum\limits_{j=0}^{n}(t_{j+1}-t_{j})\Delta(t_j)$ as $n$ goes to infinity.
\begin{center}
\begin{tikzpicture}[scale=0.9]
\draw[->, thick] (-2,0)--(4,0);
\draw[->, thick] (-1.5,-1)--(-1.5,5);
\node[above left] at (-1.6, 5) {$\R$};
\node[above right] at (-1.6, 4.9) {$t$};
\node[below] at (4, 0) {$\R^2$};
\draw[thick, red] (1,2.5) circle (2cm);
\node[above, red] at (1,4.5) {$\Delta$};
\draw[thick, violet] (-1.5,4.5)--(4,4.5);
\draw[thick, violet] (-1.5,4)--(4,4);
\draw[thick, violet] (-1.5,3.5)--(4,3.5);	
\draw[thick, violet] (-1.5,3)--(4,3);	
\draw[thick, violet] (-1.5,2.5)--(4,2.5);
\draw[thick, violet] (-1.5,2)--(4,2);
\draw[thick, violet] (-1.5,1.5)--(4,1.5);
\draw[thick, violet] (-1.5,1)--(4,1);
\draw[thick, violet] (-1.5,0.5)--(4,0.5);
\draw[thick, red] (1,2.5) circle (2cm);
\node[left,violet] at (-1.5,4.5) {$t_0$};
\node[left,violet] at (-1.5,3) {$t_i$};
\node[left,violet] at (-1.5,2.5) {$t_j$};
\node[left,violet] at (-1.5,0.5) {$t_n$};
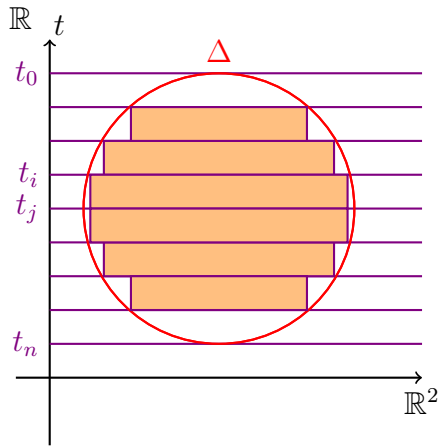
\node[below] at (0,-1.5) {{\bf Figure 2.} Minkowski integral as the limit of Riemann sums.};
\draw[thick, violet, fill=orange!50!white] (-0.7,3) rectangle (2.7,3.5);
\draw[thick, violet, fill=orange!50!white] (-0.7,1.5) rectangle (2.7,2);
\draw[thick, violet, fill=orange!50!white] (-0.9,2.5) rectangle (2.9,3);
\draw[thick, violet, fill=orange!50!white] (-0.9,2) rectangle (2.9,2.5);
\draw[thick, violet, fill=orange!50!white] (-0.3,3.5) rectangle (2.3,4);
\draw[thick, violet, fill=orange!50!white] (-0.3,1) rectangle (2.3,1.5);
\end{tikzpicture}
\end{center}
\end{rem}

\begin{theor}
In the same notation as above, for arbitrary convex bodies $B_1,\ldots,B_{n-k-1}\subset L$ the convex body $B=(k+1)!\int p|_{\Delta}\mu$ is contained in a fiber of the projection $p$ and satisfies the equality
\begin{equation}
(n-k)!\MV_{\mu'}(B,B_1,\ldots,B_{n-k-1})=n!\MV_{\mu\wedge\mu'}(\underbrace{\Delta,\ldots,\Delta}_{k+1\mbox{~times}},B_1,\ldots,B_{n-k-1}).
\end{equation}
In other words, the convex body $(k+1)!\int p|_{\Delta}\mu,$ up to a shift, is the composite body of $\underbrace{\Delta,\ldots,\Delta}_{k+1\mbox{~times}}.$
\end{theor}

\begin{sledst}
Let $\pi\colon\CC^n\to\CC^{n-k}$ be an epimorphism of complex tori, and let $f_1,\ldots,f_{k+1}$ be a generic tuple of Laurent polynomials on $\CC^n$ such that ${\newton(f_j)=\Delta,~1\leqslant j\leqslant k+1.}$ Consider the complete intersection $\tilde{X}=\{f_1=\ldots=f_{k+1}=0\}\subset\CC^n.$ Then the Newton polytope of the composite polynomial of $f_1,\ldots,f_{k+1}$ with respect to $\pi$ is equal, up to a shift, to $\int p|_{\Delta}\mu.$
\end{sledst}

\begin{theor}\label{faces_fib}
Let $\Delta\subset\R^n$ be a convex polytope, and $\gamma\in L^*$ be a covector. Then the support face $\big(\int p|_{\Delta}\mu\big)^{\gamma}$ coincides with the Minkowski sum $$\sum\limits_{\substack{\beta\in(\R^n)^*\\ \beta|_L=\gamma}}\int p|_{\Delta^{\beta}}\mu.$$
\end{theor}

Definition \ref{Mint_def} implies that if the subspace $L$ is a line with the volume form $dl,$ then the Minkowski integral $\int p|_{\Delta}\mu$ is an interval of length $n!\vol_{\mu\wedge dl}(\Delta).$ Combining this observation with Theorem \ref{faces_fib} we obtain the following result. 

\begin{lemma}\label{volumes}
Let $\Delta\subset\R^n$ be a convex polytope, $P\subset\R^n$ be a $k$-dimensional subspace, and $p$ be the projection $\R^n\to\bigslant{\R^n}{P}.$ Suppose that the face $\Gamma=\big(\int p|_{\Delta}\mu\big)^{\gamma}\subset\int p|_{\Delta}\mu$ is a segment parallel to a line $L\subset P.$ Let $\mu$ and $dl$ be volume forms on $\bigslant{\R^n}{P}$ and $L$ respectively. Then the length of $\Gamma$ equals the sum $$\sum\limits_{\substack{\beta\in(\R^n)^*\\ \beta|_L=\gamma}}\vol_{dl\wedge p^*\mu}(\Delta^{\beta}).$$
\end{lemma}
Indeed, the length of the segment $\Gamma$ in the sense of the form $dl$ is equal to the sum of volumes (in the sense of the form ${dl\wedge p^*\mu}$) of the $(n-k+1)$-dimensional faces of $\Delta$ satisfying the following property. Each of those faces is contained in an affine $(n-k+1)$-plane which intersects the plane $P$ along a line parallel to $L.$

\begin{rem}
From now on, the volume form on $\R^n$ that we will use is the lattice volume form, i.e., such that the volume of the standard unit simplex is $1.$
\end{rem}

\begin{exa}
Let $\Delta\subset\R^n$ be the standard simplex of size $d\in\Z_+,$ and $p\colon\R^3\to\R^2$ be the projection forgetting the last coordinate. Then the fiber polytope of $\Delta$ is equal to the standard $2$-dimensional simplex of size $d^2.$ 
\end{exa}

\subsection{The Euler Characteristic of a Complete Intersection}
Here we state one of the key results that we will use in this paper, namely, the formula, which expresses the Euler characteristic of a generic complete intersection in terms of the Newton polytopes of the polynomials defining it. This amazing result was obtained by A.G. Khovanskii in the work \cite{K1} (see Theorem 2, p.44). Before stating it, let us introduce some notation. 
Let $\Delta_1,\ldots,\Delta_n$ be $n$-dimensional polytopes in an $n$-dimensional space. The mixed
volume of these polyhedra is denoted by $\MV(\Delta_1,\ldots,\Delta_n)$. Now let $F(x_1,\ldots,x_k)$ be the Taylor series
of an analytic function of the $k$ variables $x_1, \ldots, x_k$ at the point $0$. We wish to determine the number $F(\Delta_1,\ldots,\Delta_k)$, and we will do is as follows: 
if $F$ is a monomial of degree $n$, namely, $x=x_1^{n_1}\cdot\ldots\cdot x_k^{n_k}, n_1 +\ldots+n_k=n$, then we put
$$F(\Delta_1,\ldots,\Delta_k)=\MV(\underbrace{\Delta_1,\ldots,\Delta_1}_{\text{$n_1$ times}},\ldots,\underbrace{\Delta_k,\ldots,\Delta_k}_{\text{$n_k$ times}}).$$
For monomials $F$ of degrees other than $n,$ we set $F(\Delta_1,\ldots,\Delta_k)=0.$ Then, by linearity, this definition is extended to arbitrary linear combinations of monomials. 

\begin{theor}\cite{K1}\label{chicomplint}
	Let $X$ be a variety defined in $\CC^n$ by a nondegenerate system of equations $f_l = \ldots =
	f_k = 0$ with Newton polyhedra $\Delta_1,\ldots, \Delta_k$. Then, in the same notation as above, we have that $$\chi(X) = \prod\Delta_i(1 + \Delta_i)^{-1}.$$ 
\end{theor}
\begin{exa}\label{chiproj} 
In the same notation as above, if $k=1,$ the variety $X$ is a hypersurface in $\CC^n$ given by a polynomial $f_1$ with the Newton polytope $\Delta_1.$ In this case, by Theorem \ref{chicomplint}, we have $$\chi(X)=(-1)^{n-1}\vol(\Delta_1).$$
\end{exa}

\begin{exa}\label{chici}
Now, suppose that $X$ is a complete intersection in $\CC^n$ of codimension $n-1$ (i.e., a curve), defined by polynomials $f_1, \ldots,f_{n-1}$ with Newton polytopes $\Delta_1,\ldots,\Delta_{n-1}.$ Then using  Theorem \ref{chicomplint}, we obtain the following formula:  $$\chi(X)=-\MV(\Delta_1,\ldots,\Delta_{n-1},\Delta_1+\ldots+\Delta_{n-1}).$$
\end{exa}

\subsection{Forking Paths Singularities}\label{intro_fps}
This subsection is devoted to the so-called {\it forking paths singularities}, introduced in the work \cite{E3}. 

Let $i=(i_1,i_2\ldots)$ be a sequence of integers satisfying the following properties: 
\begin{itemize}
	\item the sequence $i$ stabilizes at $1;$
	\item for every $r\in\N,$ the number $i_{r+1} $ divides $i_r.$ 
\end{itemize}
Given such a sequence $i,$ one can construct another sequence $q=(q_1,q_2\ldots)$ as follows: for every $r\in\N,$ set $q_{r}=\dfrac{i_r}{i_{r+1}}.$

This data can be encoded using a certain system of subsets of a finite set $R$ of $i_1$ elements via the so-called {\it $i$-nested boxes construction}. This construction works level by level as follows. Level $0$ consists of one box -- the set $R$ itself. To construct level $1,$ we divide the elements in $R$ into $q_1$ boxes containing $i_2$ elements each. Level $2$ is then the result of dividing the elements of each of the level $1$ boxes into $q_2$ boxes containing $i_3$ elements each. We continue this operation until we end up with $i_1$ boxes containing an element of $R$ each. The latter will happen in a finite number of steps, since the sequence $i$ stabilizes at $1.$ 

To illustrate the nested boxes construction, we consider the following special case. Let $R$ be the set of complex roots for a polynomial $z^{i_1}-1.$ We put elements $r_1,r_2$ into the same box on level $k,$ if $r_1^{i_{k+1}}=r_2^{i_{k+1}}.$

\begin{exa}\label{boxes}
Consider the sequence $i=(8,4,2,1,...),$ then according to the above-mentioned construction, the roots for the polynomial $z^8-1$ are placed in the following system of nested boxes: 
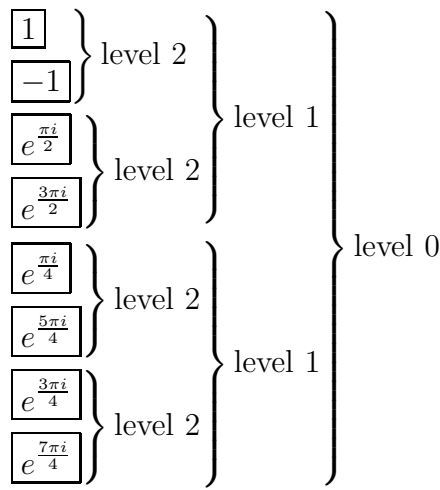
\begin{equation*}
\begin{rcases*}
\begin{rcases*}
\begin{rcases*}
\boxed{1}\\
\boxed{-1}\\
\end{rcases*}\mbox{level 2}\\
\begin{rcases*}
\boxed{e^{\frac{\pi i}{2}}}\\
\boxed{e^{\frac{3\pi i}{2}}}\\
\end{rcases*}\mbox{level 2}\\
\end{rcases*}\mbox{level 1}\\
\begin{rcases*}
\begin{rcases*}
\boxed{e^{\frac{\pi i}{4}}}\\
\boxed{e^{\frac{5\pi i}{4}}}\\
\end{rcases*}\mbox{level 2}\\
\begin{rcases*}
\boxed{e^{\frac{3\pi i}{4}}}\\
\boxed{e^{\frac{7\pi i}{4}}}\\
\end{rcases*}\mbox{level 2}\\
\end{rcases*}\mbox{level 1}\\
\end{rcases*}\mbox{level 0}\\
\end{equation*}
\end{exa}
Each of the elements of $R$ in the nested-boxes construction has its own {\it address}, i.e. a finite sequence of integers, constructed as follows. The $(k+1)-$th element of the address is the number of the $k-$th level box containing the given element. For any two elements $r_1,r_2$ of $R$ with addresses $(a_1,a_2,\ldots,a_N)$ and $(b_1,b_2,\ldots,b_N)$ one can define the {\it depth} of their relation as the number $\kappa(r_1,r_2)$ equal to the minimal number $K$ such that $a_K\neq b_K.$
\begin{exa}
If we enumerate the boxes from top to bottom on every level in Example \ref{boxes}, the element $1$ has the address $(1, 1, 1, 1),$ while the address of $e^{\frac{\pi i}{2}}$ is $(1,1,2,1).$ The depth $\kappa(1,e^{\frac{\pi i}{2}})$ of their relation is equal to $3.$
\end{exa}

\begin{defin}\label{def_fps}
In the same notation as above, let $i=(i_1,i_2,\ldots)$ be an integer sequence stabilizing at $1$ and such that $i_{r+1}$ divides $i_r$ for every $r.$ With the $i-$nested boxes construction one can associate a plane singularity with $i_1$ distinct regular branches $\varphi_{r_m}\colon(\C,0)\to(\C^2,0)$ indexed by the elements of the set $R$ and such that the intersection number of $\varphi_{r_m}$ and $\varphi_{r_n}$ with $i\neq j$ is equal to $\kappa(r_m,r_n).$ We call this singularity an {\it $i-$forking paths singularity.}
\end{defin}

\begin{utver}\label{milnornumber}\label{chifps}
The Euler characteristic $\chi(i)$ of the Milnor fiber of an $i-$forking paths singularity can be computed using the following formula:  
\begin{equation}
\chi(i)=i_1-i_1\sum\limits_{n=1}^{\infty}(i_n-1).
\end{equation}
\end{utver}

\begin{proof}
First, we independently perturb the branches of the singularity, i.e. in such a way that for every pair of branches, their multiple intersection splits into several transverse intersections. Then the union $U$ of the perturbations has $\sum\limits_{n=1}^{\infty}\dfrac{i_1(i_n-1)}{2}$ nodes. One can prove it by induction on the ``depth'' of the corresponding nested boxes construction. 

{\it Base.} For $i=(i_1,1,\ldots),$ sought number of nodes is equal to $\dfrac{i_1(i_1-1)}{2},$ since in this case, each branch transversely intersects every other branch exactly once. 

{\it Inductive step.} Let $i=(i_1,\ldots,i_k,1,\ldots)$ and $i'=(i_1,\ldots,i_k,i_{k+1},1,\ldots).$ Then the union of the perturbations of an $i'-$forking paths singularity has $\dfrac{i_1(i_{k+1}-1)}{2}=\dfrac{i_1}{i_{k+1}} \dfrac{i_{k+1}(i_{k+1}-1)}{2}$ more nodes then the one for an $i-$forking paths singularity. Indeed, we add $\dfrac{i_1}{i_{k+1}}$ boxes of $i_{k+1}$ elements each, and for every such box, the corresponding branches intersect each other transversely and exactly once, which yields $\dfrac{i_{k+1}(i_{k+1}-1)}{2}$ nodes for each of the boxes. 

The next step is to compute the Euler characteristic of the Milnor fiber of the $i-$forking paths singularity. Here we use a fact which follows from the additive property of Euler characteristic. 
\begin{utver}\label{methodmilnor}
If a perturbation $U$ of an isolated hypersurface singularity $S$ has a fiber with isolated singularities $S_1,\ldots,S_k,$ then we have:
$$\chi(\mathrm{Milnor~Fiber~of~ }S)=\chi(\mathrm{smooth~part~of~} U)+\sum\limits_{j=1}^{k}\chi(\mathrm{Milnor~ Fiber~of~}S_j).$$	
\end{utver}
In our case, the smooth part of $U$ is topologically equivalent to the union of $i_1$ complex planes with $2\sum\limits_{n=1}^{\infty}\dfrac{i_1(i_n-1)}{2}$ points punctured. The Euler characteristic of the Milnor fiber of a node is equal to $0.$
Thus we obtain the desired answer:
\begin{equation}
\chi(i)=i_1-i_1\sum\limits_{n=1}^{\infty}(i_n-1).\end{equation}
\end{proof}

\section{The Main Result}\label{mainresult}
\subsection{Dramatis Person\ae}	
\begin{itemize}
	\item[--] $(x_1,\ldots,x_n,y,t),$ coordinates in $\CC^{n+2},~n\geqslant 1;$
	\item[--] $(e_1,\ldots,e_{n+2}),$ the corresponding coordinate system in the character lattice $\Z^{n+2};$
	\item[--] $A\subset\Z^{n+2},$ a finite subset of maximal dimension;
	\item[--] $\Delta=\conv(A)\subset\R^{n+2},$ the convex hull of the set $A;$
	\item[--] $\mathcal{F}(\Delta),$ the set of all facets of the polytope $\Delta;$
	\item[--] $X_{\Delta},$ the toric variety associated to the polytope $\Delta;$
	\item[--] $f_1,\ldots,f_{n+1}\in\C^A,$ a tuple of polynomials supported at $A;$
	\item[--] $\tilde{\mathcal C}=\{f_1=\ldots=f_{n+1}=0\}\subset\CC^{n+2},$ the complete intersection given by the polynomials $f_1,\ldots,f_{n+1};$
	\item[--] $\pi\colon\CC^{n+2}\to\CC^2,$ the projection forgetting the first $n$ coordinates;
	\item[--] $\mathcal C\subset\CC^2,$ the closure of the image $\pi(\tilde{\mathcal{C}})\subset\CC^2;$
	\item[--] $P\subset\R^2$ the Newton polygon of the curve $\mathcal C;$
	\item[--] $\mathcal S,$ the singular locus of the curve ${\mathcal C}.$

\end{itemize}
\subsection{Statement of the Problem}
In this subsection we give a precise formulation of the question that we address in this paper and discuss all the assumptions we make. For generic $f_1,\ldots,f_{n+1}\in\C^A,$ the complete intersection $\tilde{\mathcal C}=\{f_1=\ldots=f_{n+1}=0\}\subset\CC^{n+2}$ is a smooth curve and the closure $\mathcal C$ of its image under the projection $\pi$ is a plane curve in $\CC^2,$ whose singular locus consists of finitely many isolated singular points. 

It is quite natural to expect that under certain genericity conditions, all the singular points of the curve $\mathcal C$ are nodes. However, it is the case not for all support sets. Moreover, the following example shows that when the support sets $\supp(f_j)=A_j$ do not coincide, one can no longer expect the singular points of the projection of the corresponding complete intersection to be nodes even for generic polynomials $f_j\in\C^{A_j}.$

\begin{exa}\label{notnodal}
Consider $A_1=\{0,1,3\}\subset \Z^1$ and $A_2=\{0,3\}\subset\Z^1.$ Let $f_1(x)$ and $f_2(x)$ be polynomials supported at $A_1$ and $A_2$ respectively. Suppose that the univariate system  $\{f_1(x)=f_2(x)=0\}$ has $2$ distinct roots $r_1,r_2\in\CC.$ Let us show that this system also has another root $r_3\in\CC.$ Indeed, the assumption we made implies that $r_2=\alpha\cdot r_1,$ where $\alpha$ is a root of unity. Substituting these roots into the first equation, we obtain that the linear term of $f_1$ has to be $0.$ But then it is clear that the third root $r_3=\alpha\cdot r_2$ is also a root for the first equation. 
\end{exa}

Therefore, in this paper we address a slightly more general question: namely, we compute the sum of the $\delta$-invariants (for details see \S 10 in \cite{Milnor}) of the singular points of the curve $\mathcal{C}.$ On one hand, this question makes sense for any support sets. On the other hand, if the curve $\mathcal{C}$ only has nodes as singularities, then the answer is exactly the number of those nodes.

\begin{rem}\label{deltarem}
Roughly speaking, the $\delta$-invariant of a plane curve singularity $S$ is a non-negative integer that is equal to the number of double points concentrated at it. Let $S$ be a singularity with $b(S)$ branches. Perturb its branches in such a way that the union $U$ of the perturbations has only nodes as singularities. Denote by $N$ the number of these nodes. By Proposition \ref{methodmilnor}, the Euler characteristic of the Milnor fiber of $S$ is then equal to $b(S)-2N.$ Therefore, the Milnor number of $S$ is given by the following formula: $$\mu(S)=1-b(S)+2N,$$ or, equivalently, $$2N=\mu(s)+b(S)-1.$$ 
The latter then agrees with Milnor formula $2\delta(s)=\mu(s)+b(S)-1,$ which relates the	$\delta$-invariant of the singularity, its Milnor number and the number of its  branches.
\end{rem}

\begin{prb}\label{mainprb}
In the same notation as above, express the sum $\mathcal D$ of the $\delta$-invariants of the singular points of the curve $\mathcal{C}$ in terms of the set $A.$ 
\end{prb}

Let us introduce some more notation that will be used a lot throughout the paper. Let $\tilde{\Lambda}_A\subset\Z^{n+2}$ be the sublattice generated by $A,$ and let $\Lambda_A$ be its image under the projection $\rho\colon\Z^{n+2}\twoheadrightarrow\bigslant{Z^{n+2}}{\langle e_{n+1},e_{n+2}\rangle}.$ Then by $\ind_v(A)$ we denote the index of $\Lambda_A$ in $\bigslant{Z^{n+2}}{\langle e_{n+1},e_{n+2}\rangle}.$

\begin{predpol}
The set $A$ contains $0\in\Z^{n+2}.$
\end{predpol}
This assumption can be made since multiplication by monomial does not change the zero set of the polynomial inside the algebraic torus. At the same time, the resulting support set is a shift of the initial one.
\begin{predpol}\label{indall}
The set $A$ satisfies the following property: $ind_v(A)=1.$
\end{predpol}

\begin{rem}\label{indv}
We can make this assumption due to the following reason. $\Lambda_A$ and $\Lambda=\bigslant{Z^{n+2}}{\langle e_{n+1},e_{n+2}\rangle}$ admit a pair of aligned bases such that $\Lambda=\bigoplus\Z w_i$ and $\Lambda_A=\bigoplus\Z a_i w_i$ for some $a_i\in\Z.$ Performing a monomial change of variables to pass from the basis $(e_1,\ldots,e_{n+2})$ to $(w_1,\ldots,w_{n},e_{n+1},e_{n+2})$ and then another change of variables of the form $\check{x}_i=x_i^{a_i},$ we will reduce our problem to the case $ind_v(A)=1.$	
\end{rem} 

Let $Q=p(\Delta)$ be the image of the polytope $\Delta$ under the projection $\rho\colon\R^{(n+2)}\twoheadrightarrow\bigslant{\R^{(n+2)}}{\langle e_{n+1},e_{n+2}\rangle}.$ 

\begin{defin}
We call a face $\tilde{\Gamma}\subset\Delta$ {\it horizontal}, if its projection is contained in the boundary of $Q$. We denote the set of all horizontal facets of the polytope $\Delta$ by $\mathcal{H}(\Delta).$
\end{defin}

\subsection{Statement of the Main Result}
Let $\Gamma\subset\Delta$ be a non-horizontal facet contained in a hyperplane given by a linear equation of the form $\ell(e_1,\ldots,e_{n+2})=c.$ The function $\ell$ is unique up to a scalar multiple, therefore, one can assume that the coefficients of $\ell$ are coprime integers and that for any $\alpha\in A\setminus\Gamma,~\ell(\alpha)<c.$ 

We now construct a sequence of integers $i^{\Gamma}=(i_1^{\Gamma},i_2^{\Gamma},\ldots)$ as follows. 

Set $B_1^{\Gamma}=A\cap\Gamma.$ For every $r>1,$ we define  $$B_r^{\Gamma}=B_{r-1}^{\Gamma}\cup(A\cap\{\ell(e_1,\ldots,e_{n+2})=c- (r-1)\}),$$ depending on the way $\Delta$ is positioned relative to the hyperplane containing $\Gamma.$ 
Finally, for every $r\geqslant 1,$ we set $$i_r^{\Gamma}=\ind_{v}(B_r^{\Gamma}).$$

It is clear that for every $r,$ the element $i_r^{\Gamma}$ divides $i_{r-1}^{\Gamma}.$ Moreover, since for the set $A$ we have $\ind_v(A)=1,$ any such sequence stabilizes to $1$.

\begin{theor}\label{lemmain}
	Let $A\subset\Z^{n+2}$ be a finite set of full dimension, satisfying Assumption \ref{indall}, and let $\Delta\subset\R^{n+2}$ be its convex hull. In the same notation as above, for generic $f_1,\ldots,f_{n+1}\in\C^A,$ the closure $\mathcal C$ of the image of the curve $\tilde{\mathcal{C}}=\{f_1=\ldots=f_{n+1}=0\}$ under the projection $\pi\colon\CC^{n+2}\to\CC^2$ forgetting the first $n$ coordinates is an algebraic plane curve, whose singular locus $\mathcal S$ consists of isolated singular points. 
	Then the number $\mathcal D=\sum\limits_{s\in\mathcal S}\delta(s)$ can be computed via the following formula:  
	\begin{equation}\label{mainformula1}
	\mathcal D=\dfrac{1}{2}\Bigg(\area(P)-(n+1)\vol(\Delta)+\sum\limits_{\Gamma\in\mathcal H(\Delta)}\vol(\Gamma)-\sum\limits_{\Gamma\in\mathcal{F}(\Delta)\setminus\mathcal{H}(\Delta)}\vol(\Gamma)\sum\limits_1^{\infty}(i_r^{\Gamma}-1)\Bigg),
	\end{equation}
	where $\delta(s)$ is the $\delta$-invariant of the singular point $s$, $P=\int_{\pi}(\Delta),$ the set $\mathcal{F}(\Delta)$ is the set of all facets of the polytope $\Delta$ and $\mathcal{H}(\Delta)$ is the set of all horizontal facets of $\Delta.$
\end{theor}    

\section{Proof of the Main Result}\label{mainproof}
This Section is organized as follows. 

Subsections \ref{s1} and \ref{s2} are devoted to the singular points of the curve $\mathcal C$ at infinity. There we show that for generic $f_1,\ldots,f_{n+1}\in\C^A,$ all the singularities of the curve $\mathcal C$ at infinity are forking path singularities (see Subsection \ref{intro_fps} for details) and compute their contribution to the Euler characteristic of the curve $\mathcal{C}.$ In Subsection \ref{main}, we compare the Euler characteristic of the curves $\tilde{\mathcal{C}}$ and $\mathcal{C}$ to deduce the main result of this paper -- Theorem \ref{lemmain}.

\subsection{Order of Contact Between Branches of the Curve $\mathcal C$ at Infinity}\label{s1}
We first make a very useful technical assumption that will simplify the proof of Theorem \ref{lemmain}.

\begin{predpol}\label{primitive}
	For every primitive covector $\gamma$ such that $\Delta^{\gamma}\subset\Delta$ is a facet, its image under the projection forgetting the first $n$ coordinates is also primitive. 
\end{predpol}

\begin{exa}\label{simpl1}
	Consider ${A=\{(0,0,0),(0,1,0),(0,0,1),(1,0,0),(2,0,0)\}}\subset\Z^3$ (see Figure 3 below).
	The covector $\gamma=(1,2,2)\in(\R^3)^*$ is supported at the facet $\Gamma_0=\conv(\{(0,1,0),(0,0,1),(2,0,0)\})$ (hatched orange) of the polytope $\Delta=\conv(A).$ The covector $\gamma$ is primitive, while its projection $(2,2)$ is not. At the same time, the projection of the primitive covector supported at the facet $\Gamma_1=\conv(\{(0,0,0),(0,0,1),(2,0,0)\})$ (hatched blue) is primitive. 
	\begin{center}
		\begin{tikzpicture}[scale=0.9]
		\pattern[pattern=north east lines, pattern color=orange]  (1,0)--(2,1)--(0,5)--(1,0);
		\pattern[pattern=north west lines, pattern color=blue]  (0,1)--(1,0)--(0,5)--(0,1);
		\draw[ultra thick] (0,1)--(0,5);
		\draw[ultra thick] (0,1)--(1,0);
		\draw[ultra thick] (0,5)--(1,0);
		\draw[ultra thick] (0,5)--(2,1);
		\draw[ultra thick] (2,1)--(1,0);
		\draw[dashed, ultra thick] (0,1)--(2,1);
		\draw[fill,red] (0,1) circle [radius=0.1];
		\draw[fill,red] (2,1) circle [radius=0.1];
		\draw[fill,red] (1,0) circle [radius=0.1];
		\draw[fill,red] (0,5) circle [radius=0.1];
		\draw[fill,red] (0,3) circle [radius=0.1];
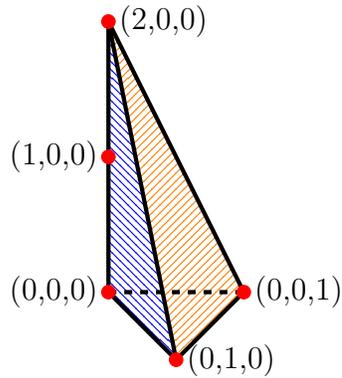
		\node[below] at (1,-0.3) {{\bf Figure 3.} A polytope that does not satisfy Assumption \ref{primitive}.};
		\node[right] at (0,5) {(2,0,0)};
		\node[left] at (0,3) {(1,0,0)};
		\node[left] at (0,1) {(0,0,0)};
		\node[right] at (2,1) {(0,0,1)};
		\node[right] at (1,0) {(0,1,0)};
		\end{tikzpicture}
	\end{center}
\end{exa}

\begin{rem}\label{primitive_exp}
We make this purely technical assumption to be able to perform certain monomial changes of variables that would preserve horizontality of facets in $\Delta.$ Moreover, as we will further see, this assumption guarantees that the projection $\mathcal C$ of the complete intersection $\tilde{\mathcal C}$ has only forking--path singularities at infinity (see Lemma \ref{fpsinfinity}  for details). Also note that one can always reduce the computation to the case when the polytope $\Delta$ satisfies Assumption \ref{primitive}. For instance, consider a monomial change of variables $\check{x}_i=x_i,~\check{y}^{N!}=y,~\check{t}^{N!}=t,$ for $N=k!,$ where $k=\max_{\Gamma\subset\Delta}(\ind_v(A\cap\Gamma)).$ The polytope $\check{\Delta}$ clearly satisfies the desired condition. 
	
At the same time, inside the torus $\CC^2,$ the projection of the new complete intersection $\check{\mathcal C}$ defined by polynomials $\check{f_1},\ldots,\check{f}_{n+1}$ with the Newton polytope $\check{\Delta}$ can be viewed as an $(N!)^2$-covering of the initial curve $\mathcal C$, therefore, to find the number $\mathcal D$ for the curve $\mathcal C$, one can just find this number for the projection of the curve $\check{\mathcal C}$ and divide it by $(N!)^2$. 
	
Moreover, we will later see that the answer does not depend on the choice of $N,$ therefore our main result (Theorem \ref{lemmain}) holds true without this assumption.
\end{rem}

\begin{exa}\label{simp2}
	Let us come back to $\Delta$ from Example \ref{simpl1}. The monomial change $\check{x}_1=x_1,~\check{t}^{2!}=t,~\check{y}^{2!}=y$ reduces our problem to a much easier case $\check{\Delta}=\conv(\{(0,0,0),(2,0,0),(0,2,0),(1,0,0),(0,0,2)\})$ (see Figure 4).
	\begin{center}
		\begin{tikzpicture}
		\pattern[pattern=north east lines, pattern color=orange] (1,0)--(2,1)--(0,5)--(1,0);
		\pattern[pattern=north east lines, pattern color=orange] (8,-1)--(10,1)--(6,5)--(8,-1);
		\draw[ultra thick] (0,1)--(0,5);
		\draw[ultra thick] (0,1)--(1,0);
		\draw[ultra thick] (0,5)--(1,0);
		\draw[ultra thick] (0,5)--(2,1);
		\draw[ultra thick] (2,1)--(1,0);
		\draw[ultra thick] (6,1)--(6,5);
		\draw[ultra thick] (6,1)--(8,-1);
		\draw[ultra thick] (6,5)--(8,-1);
		\draw[ultra thick] (6,5)--(10,1);
		\draw[ultra thick] (10,1)--(8,-1);
		\draw[dashed, ultra thick] (0,1)--(2,1);
		\draw[dashed, ultra thick] (6,1)--(10,1);
		\draw[fill,red] (0,1) circle [radius=0.1];
		\draw[fill,red] (2,1) circle [radius=0.1];
		\draw[fill,red] (1,0) circle [radius=0.1];
		\draw[fill,red] (0,5) circle [radius=0.1];
		\draw[fill,red] (0,3) circle [radius=0.1];
		\draw[fill,red] (6,1) circle [radius=0.1];
		\draw[fill,red] (10,1) circle [radius=0.1];
		\draw[fill,red] (8,-1) circle [radius=0.1];
		\draw[fill,red] (6,5) circle [radius=0.1];
		\draw[fill,red] (6,3) circle [radius=0.1];
		\draw [ultra thick, fill=white] (8,1) circle [radius=0.1];
		\draw [ultra thick, fill=white] (8,3) circle [radius=0.1];
		\draw [ultra thick, fill=white] (9,0) circle [radius=0.1];
		\draw [ultra thick, fill=white] (7,2) circle [radius=0.1];
		\draw [ultra thick, fill=white] (7,0) circle [radius=0.1];
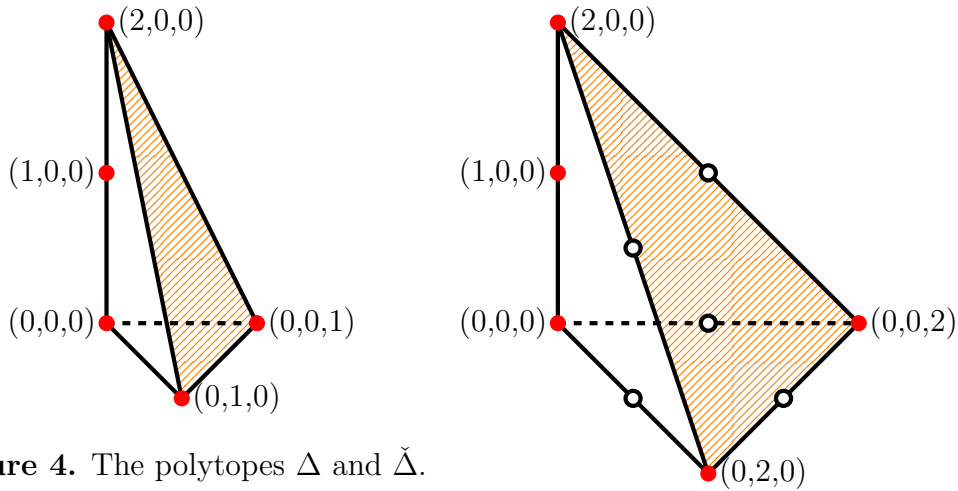
		\node[below] at (1,-0.5) {{\bf Figure 4.} The polytopes $\Delta$ and $\check{\Delta}.$};
		\node[right] at (0,5) {(2,0,0)};
		\node[left] at (0,3) {(1,0,0)};
		\node[left] at (0,1) {(0,0,0)};
		\node[right] at (2,1) {(0,0,1)};
		\node[right] at (1,0) {(0,1,0)};
		\node[right] at (6,5) {(2,0,0)};
		\node[left] at (6,3) {(1,0,0)};
		\node[left] at (6,1) {(0,0,0)};
		\node[right] at (10,1) {(0,0,2)};
		\node[right] at (8,-1) {(0,2,0)};
		\end{tikzpicture}
	\end{center}
	After the change of variables, the primitive covector supported at $\check{\Gamma}$ (hatched orange) is equal to $(1,1,1),$ therefore, its projection $(1,1)$ is also primitive. So, Assumption \ref{primitive} is now satisfied.
\end{exa}

\begin{utver}\label{vert_change}
Let $\Gamma=\Delta^{\gamma}\subset\Delta$ be a non-horizontal facet. Then, under Assumption \ref{primitive}, there exists a basis $(h_1,\ldots,h_{n+2})$ in $\Z^{n+2}$ satisfying the following conditions: 
\begin{itemize}
\item passing from the basis $(e_1,\ldots,e_{n+2})$ to $(h_1,\ldots,h_{n+2})$ preserves the set $\mathcal{H}(\Delta)$ of the horizontal facets of $\Delta;$
\item in the basis $(h_1,\ldots,h_{n+2})$ the facet $\Gamma$ is ``vertical'', i.e., parallel to the coordinate hyperplane $\{h_{n+2}=0\}.$ 
\end{itemize}
\end{utver}

\begin{proof}
 Indeed, let $(\check{e}_1,\ldots,\check{e}_{n+2})$ be the basis of $(\Z^{n+2})^{*}$ dual to the basis $(e_1,\ldots,e_{n+2})$ of the lattice $\Z^{n+2}.$ Consider the primitive covector $\gamma\in(\Z^{n+2})^{*}$ that is supported at the facet $\tilde{\Gamma}.$ Assumption \ref{primitive} implies that the image of $\gamma$ under the projection forgetting the first $n$ coordinates is primitive, therefore, together with some integer covector $\tau=\lambda_1\check{e}_{n+1}+\lambda_2\check{e}_{n+2},$ it forms a basis of the $2$-dimensional lattice $\langle\check{e}_{n+1},\check{e}_{n+2}\rangle. $ We construct a new basis $(\check{h}_1,\ldots,\check{h}_{n+2})$ of $(\Z^{n+2})^*$ as follows. We set 
$$
\check{h}_i=
\begin{cases}
\check{e}_i,~\text{ if } 1\leqslant i\leqslant n\\
\tau~\text{ if } i=n+1\\
\gamma~\text{ if } i=n+2.
\end{cases}
$$
Let $M$ be the corresponding transition matrix. The monomial change of variables given by its transpose $M^T$ does preserve horizontality of faces of $\Delta.$ Moreover, under this change of variables the facet $\tilde{\Gamma}$ becomes ``vertical'', since in the new coordinate system $(h_1,\ldots, h_{n+2})$ it is parallel to the coordinate hyperplane $\{h_{n+2}=0\}.$ 
\end{proof}

\begin{exa}\label{simp3}
Consider the support set $\check{A}=\{(0,0,0),(1,0,0),(2,0,0),(0,2,0),(0,0,2)\}$ and the polytope $\check{\Delta}=\conv{\check{A}}$ from Example \ref{simp2} (see Figure 5 below). The monomial change of variables $x\mapsto\tilde{x}\tilde{t},~y\mapsto\tilde{t},~t\mapsto \tilde{y}\tilde{t}$ turns a polynomial of the form $f(x,y,t)=a+by^2+ct^2+dx+ex^2$ into the polynomial $\tilde{f}(\tilde{x},\tilde{y},\tilde{t})=a+b\tilde{t}^2+c\tilde{y}^2\tilde{t^2}+d\tilde{x}\tilde{t}+e\tilde{x}^2\tilde{t}^2.$ Under this change of variables the facet $\Gamma=\conv(\{(2,0,0),(0,2,0),(0,0,2)\})\subset\check{\Delta}$ becomes a ``vertical'' facet $\tilde{\Gamma}=\conv(\{(2,0,2),(0,0,2),(0,2,2)\})\subset\tilde{\Delta}.$
\begin{center}
\begin{tikzpicture}
\pattern[pattern=north east lines, pattern color=orange] (2,-1)--(4,1)--(0,5)--(2,-1);
\pattern[pattern=north west lines, pattern color=orange] (12,1)--(14,-1)--(12,5)--(12,1);
\draw[dashed, ultra thick] (0,1)--(4,1);
\draw[ultra thick] (0,1)--(0,5);
\draw[ultra thick] (0,1)--(2,-1);
\draw[ultra thick] (0,5)--(2,-1);
\draw[ultra thick] (0,5)--(4,1);
\draw[ultra thick] (4,1)--(2,-1);
\draw[ultra thick] (8,1)--(14,-1);
\draw[ultra thick] (8,1)--(12,5);
\draw[ultra thick] (12,5)--(14,-1);
\draw[dashed, ultra thick] (8,1)--(12,1);
\draw[dashed, ultra thick] (12,1)--(14,-1);
\draw[dashed, ultra thick] (12,1)--(12,5);
\draw[fill,red] (0,1) circle [radius=0.1];
\draw[fill,red] (4,1) circle [radius=0.1];
\draw[fill,red] (2,-1) circle [radius=0.1];
\draw[fill,red] (0,5) circle [radius=0.1];
\draw[fill,red] (0,3) circle [radius=0.1];
\draw[fill,red] (8,1) circle [radius=0.1];
\draw[fill,red] (12,1) circle [radius=0.1];
\draw[fill,red] (12,5) circle [radius=0.1];
\draw[fill,red] (14,-1) circle [radius=0.1];
\draw[fill,red] (10,3) circle [radius=0.1];
\draw [ultra thick, fill=white] (2,1) circle [radius=0.1];
\draw [ultra thick, fill=white] (2,3) circle [radius=0.1];
\draw [ultra thick, fill=white] (3,0) circle [radius=0.1];
\draw [ultra thick, fill=white] (1,2) circle [radius=0.1];
\draw [ultra thick, fill=white] (1,0) circle [radius=0.1];
\draw [ultra thick, fill=white] (10,1) circle [radius=0.1];
\draw [ultra thick, fill=white] (13,0) circle [radius=0.1];
\draw [ultra thick, fill=white] (11,0) circle [radius=0.1];
\draw [ultra thick, fill=white] (12,3) circle [radius=0.1];
\draw [ultra thick, fill=white] (13,2) circle [radius=0.1];
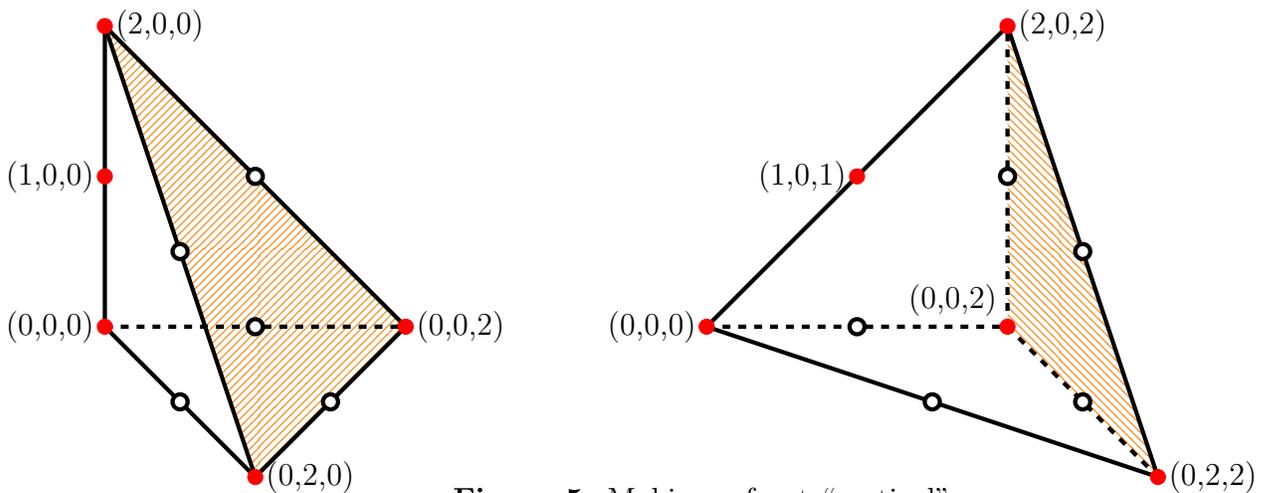
\node[below] at (8,-1) {{\bf Figure 5.} Making a facet ``vertical''.};
\node[right] at (0,5) {(2,0,0)};
\node[left] at (0,3) {(1,0,0)};
\node[left] at (0,1) {(0,0,0)};
\node[right] at (4,1) {(0,0,2)};
\node[right] at (2,-1) {(0,2,0)};
\node[above left] at (12,1) {(0,0,2)};
\node[left] at (8,1) {(0,0,0)};
\node[left] at (10,3) {(1,0,1)};
\node[right] at (12,5) {(2,0,2)};
\node[right] at (14,-1) {(0,2,2)};
\end{tikzpicture}
\end{center}
\end{exa}

\begin{lemma}\label{inf_br}
For generic $f_1,\ldots, f_{n+1}\in\C^A,$ every limiting point $p$ of the curve $\mathcal C$ in the toric variety $X_P$ is contained in the $1$-dimensional orbit $\mathcal O_{\Gamma}$ associated to some edge $\Gamma\subset P.$ Moreover, under Assumption \ref{primitive}, in a small neighborhood of $p$ every branch of the curve $\mathcal C$ is a smooth curve that intersects $\mathcal O_{\Gamma}$ transversally. 
\end{lemma}

\begin{proof}
The first statement is a corollary of Theorem 4.10 in \cite{EK}. Moreover, this result implies that in a small neighborhood of $p,$ each branch of the curve $\mathcal C$ is the projection of the intersection of the curve $\tilde{\mathcal C}$ with a small neighborhood of some limiting point $q$ in the orbit $\mathcal O_{\tilde{\Gamma}}\subset X_{\Delta},$ where $\tilde{\Gamma}\subset \Delta$ is some facet parallel to the edge $\Gamma.$ It is well known that for generic $f_1,\ldots, f_{n+1}\in\C^A,$ the complete intersection $\tilde{\mathcal C}$ in a small neighborhood of $q$ is smooth and intersects $\mathcal O_{\tilde{\Gamma}}$ transversally. 

Let us now assume that the edge $\Gamma\subset P$ is contained in a coordinate axis, and the facet $\tilde{\Gamma}$ is ``vertical'' and is contained in the coordinate hyperplane whose image under the projection is that coordinate axis. In this setting, the torus $\CC^2\subset X_{P}$ and the orbit $\mathcal O_{\Gamma}$ can be considered as dense subsets of the plane $\C^2$ together with one of its coordinate axes respectively, while the torus $\CC^{n+2}\subset X_{\Delta}$ and the orbit $\mathcal O_{\tilde{\Gamma}}$ can be considered as dense subsets of $\C^{n+2}$ together with one of its coordinate hyperplanes. Therefore, $\pi$ is just the projection from $\C^{n+2}$ to $\C^2$ forgetting the first $n$ coordinates, which maps the torus $\CC^{n+2}$ to the torus $\CC^2,$ the orbit $\mathcal O_{\tilde{\Gamma}}$ to the orbit  $\mathcal O_{\Gamma}$ and the point $q$ to the point $p.$ Thus, under this assumption, the statement of the lemma is reduced to the following trivial statement:
if a curve $C\subset \C^{n+2}$ is smooth and intersects a ``vertical'' coordinate hyperplane transversally at the point $q,$ then its image $\pi(C)$ is a smooth curve which transversally intersects the corresponding coordinate axis at the point $p=\pi(q).$ 

Finally, due to Assumption \ref{primitive}, the general statement of the lemma can be reduced to the special case considered above by applying a certain monomial change of variables preserving horizontality of facets of $\Delta.$ The existence of such a change of variables follows from Proposition \ref{vert_change}.
\end{proof}

\begin{exa}
	Let $\pi\colon\CC^2\to\CC$ be the projection given by $\pi\colon (x,y)\mapsto y,$ and consider the set $A=\{(0,0),(2,0),(0,1)\}.$ Its convex hull $\Delta$ does not satisfy Assumption \ref{primitive}. Take a polynomial $f$ with $\supp(f)=A.$ Up to a scalar multiple, its truncation $f^{\Gamma}$ where $\Gamma=\conv(\{(2,0),(0,1)\})$ is a polynomial of the form $\alpha x^2-y,$ so the intersection of the branches of the curve $C$ defined by $f$ and the $1-$dimensional orbit of $X_{\Delta}$ corresponding to $\Gamma$ is not transverse. Now, consider the change of variables $\check{x}=x, \check{y}=y^2.$ Then the truncation $\check{f}^{\check{\Gamma}}$ after this change of variables is the polynomial $\alpha\check{x}^2-\check{y}^2,$ whose zero set is a union of two lines $\{\sqrt{\alpha}\check{x}\pm\check{y}\}$ that do intersect the corresponding orbit of the toric variety $X_{\check{\Delta}}$ transversally.  Note that, inside the torus $\CC^2,$ the curve $\check{C}$ defined by $\check{f}$ is a $2$-covering of the curve $C$ defined by $f.$
\end{exa}

Consider a $1$-dimensional orbit $\mathcal O_{\Gamma}$ of the toric variety $X_P$ containing multiple points of the closure of the curve $\mathcal{C}.$ Let $\Gamma\subset P$ be the corresponding edge of the polygon $P$ and $\gamma\in(\Z^2)^*$ be the primitive covector supported at $\Gamma.$ 

\begin{predpol}\label{facets}
For any pair of facets $\tilde{\Gamma}=\Delta^{\gamma_1},\tilde{\Gamma}'=\Delta^{\gamma_2},~\gamma_1\neq\gamma_2$ parallel to the edge $\Gamma$, we have $$\pi(\{f_1^{\tilde{\Gamma}}=\ldots=f_{n+1}^{\tilde{\Gamma}}=0\})\cap\pi(\{f_1^{\tilde{\Gamma}'}=\ldots=f_{n+1}^{\tilde{\Gamma}'}=0\})=\varnothing.$$
\end{predpol}

\begin{utver}\label{diff_facets}
The tuples of polynomials $f_1,\ldots,f_{n+1}\in\C^A$ satisfying Assumption \ref{facets} form an everywhere dense Zariski open subset.
\end{utver}

\begin{proof}
It is clear that this set is Zariski open. Therefore, one just has to show that it is not empty. Indeed, take a tuple of polynomials that does not satisfy the assumption. Since  $\tilde{\Gamma}\neq\tilde{\Gamma'},$ there exists a monomial $f$ belonging to the set $A\cap\tilde{\Gamma}$ but not to the set $\tilde{\Gamma'}.$ Perturbing its coefficient in one of the polynomials $f_j$, for instance, in $f_{1},$ we do not change any of the roots for the truncated system $\{f_1^{\tilde{\Gamma}'}=\ldots=f_{n+1}^{\tilde{\Gamma}'}=0\}.$ At the same time, this system has only finitely many roots, and, thus, for generic $\varepsilon,$ the projections of the roots for the system $\{f_1^{\tilde{\Gamma}}+\varepsilon\cdot f=\ldots=f_j^{\tilde{\Gamma}}=\ldots=f_{n+1}^{\tilde{\Gamma}}=0\}$ do not coincide with those of the roots for $\{f_1^{\tilde{\Gamma}'}=\ldots=f_{n+1}^{\tilde{\Gamma}'}=0\}.$
\end{proof}

The integer length of an edge $\Gamma=P^{\gamma}\subset P$ equals the sum of the integer volumes of all the facets $\tilde{\Gamma}\subset\Delta$ that are parallel to $\Gamma$  (see Lemma \ref{volumes} for details). 
Proposition \ref{diff_facets} implies that, for generic $f_1,\ldots,f_{n+1}\in\C^A,$ each of these facets contributes to its own subset of multiple points in the orbit $\mathcal O_{\Gamma}\subset X_{P},$ and, for any pair of facets $\tilde{\Gamma}_1, \tilde{\Gamma}_2$ parallel to $\Gamma,$ the corresponding subsets do not intersect. Therefore, without loss of genericity, one can consider the case when there is a unique facet $\tilde{\Gamma}\subset\Delta,$ such that the projection of the primitive covector supported at it coincides with $\gamma.$ In this setting, the polynomial $h^{\Gamma}$ has $\vol(\tilde{\Gamma})$ roots if counted with multiplicities, and those multiplicities are equal to $\ind_v(A\cap\tilde{\Gamma})$  (for details, see Theorem 5.3 in \cite{M}).

Let us also suppose that the facet $\tilde{\Gamma}$ is ``vertical'', i.e., the restrictions $f_1^{\tilde{\Gamma}},\ldots,f_{n+1}^{\tilde{\Gamma}}$ are polynomials in $x_1,\ldots,x_n, y.$ Due to Assumption \ref{primitive}, one can always achieve that by performing a certain monomial change of variables preserving horizontality of faces of the polytope $\Delta$ (see Proposition \ref{vert_change}).

Let $p$ be a multiple point of intersecton of the closure of the curve $\mathcal C$ with the $1$-dimensional orbit $\mathcal O_{\Gamma}$ of the toric variety $X_{P}$ corresponding to the edge $\Gamma.$ Lemma \ref{inf_br} implies that in a small neighborhood of the point $p,$ the curve $\mathcal C$ is a union of smooth branches intersecting $\mathcal O_{\Gamma}$ transversally. 
We will now compute the order of contact between any two of those branches in terms of the support set $A.$ 

In the special case we are now considering, the polynomials $f_i,~1\leqslant\ldots i\leqslant n+1,$ can be written in the following form: $f_i=g_i(x_1\ldots,x_n,y)+\sum_{t=1}^{\infty} t^m\tilde{g}_{i,m}(x_1\ldots,x_n,y),$ where $g_i=f_i^{\tilde{\Gamma}}(x_1,\ldots,x_n,y),$ and $\tilde{g}_{i,m}$ are polynomials in the variables $x_1,\ldots,x_n,y.$ 

\begin{exa}\label{slice}
Consider the support set $A=\{(0,0,0),(1,0,0),(2,0,0),(0,2,0),(0,0,2)\}$ and the polytope $\Delta=\conv(A)$ (see Figure 6 below).
 
\begin{center}
	\begin{tikzpicture}
	\pattern[pattern=north west lines, pattern color=orange] (0,1)--(2,-1)--(0,5)--(0,1);
	\draw[ultra thick,orange] (0,1)--(0,5);
	\draw[ultra thick,orange] (0,1)--(2,-1);
	\draw[ultra thick,orange] (0,5)--(2,-1);
	\pattern[pattern=north west lines, pattern color=orange] (2,3)--(3,0)--(2,1)--(2,3);
	\draw[ultra thick,orange] (2,3)--(3,0);
	\draw[dashed, ultra thick,orange] (2,3)--(2,1);
	\draw[dashed, ultra thick,orange] (2,1)--(3,0);
	\draw[ultra thick] (0,5)--(4,1);
	\draw[ultra thick] (4,1)--(2,-1);
	\draw[dashed, ultra thick] (0,1)--(4,1);
	\draw[fill,red] (0,1) circle [radius=0.1];
	\draw[fill,red] (4,1) circle [radius=0.1];
	\draw[fill,red] (2,-1) circle [radius=0.1];
	\draw[fill,red] (0,5) circle [radius=0.1];
	\draw[fill,red] (2,3) circle [radius=0.1];
	\draw [ultra thick, fill=white] (2,1) circle [radius=0.1];
	\draw [ultra thick, fill=white] (0,3) circle [radius=0.1];
	\draw [ultra thick, fill=white] (3,0) circle [radius=0.1];
	\draw [ultra thick, fill=white] (1,2) circle [radius=0.1];
	\draw [ultra thick, fill=white] (1,0) circle [radius=0.1];
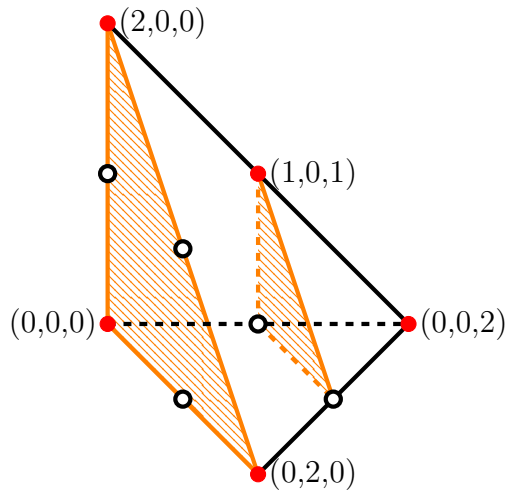
	\node[below] at (2,-1.3) {{\bf Figure 6.}  Slicing the polytope $\Delta$.};
	\node[right] at (0,5) {(2,0,0)};
	\node[right] at (2,3) {(1,0,1)};
	\node[left] at (0,1) {(0,0,0)};
	\node[right] at (4,1) {(0,0,2)};
	\node[right] at (2,-1) {(0,2,0)};
	\end{tikzpicture}
\end{center}
A polynomial $f_j(x_1,y,t)$ supported at the set $A,$ can be written in the form $f_j(x_1,y,t)=f_j^{\tilde{\Gamma}}(x_1,y)+t\cdot\tilde{g}_{j,1}(x_1,y)+t^2\cdot\tilde{g}_{j,2}(x_1,y)=(a+bx_1^2+cy^2)+t\cdot \alpha x_1y+t^2\cdot \beta.$
\end{exa}
 
Let $p=(0,\ldots,0,u,0)$ be a multiple intersection point of the curve $\mathcal{C}$ with the $1$-dimensional orbit of the toric variety $X_{P}$ corresponding to the edge $\Gamma\subset P.$ Consider a pair of branches of the curve $\mathcal{C}$ passing through the point $p.$ Then the preimages of this two branches of the curve $\mathcal C$ intersect the orbit corresponding to the facet $\tilde{\Gamma}\subset\Delta$ at the points $p_1=(v_1,\ldots,v_n,u,0), p_2=(v'_1,\ldots,v'_n,u,0)$ respectively. Moreover, for every $1\leqslant i\leqslant n+1,$ we have $g_i(p_1)=g_i(p_2).$ Due to Assumption \ref{indall} the support set $A$ satisfies $\ind_v(A)=1$. Therefore there exists a number $K\in\N,$ such that for $1\leqslant i\leqslant n+1,$ we have $\tilde{g}_{i,K-1}(p_1)=\tilde{g}_{i,K-1}(p_2)$ and $\tilde{g}_{i,K}(p_1)\neq\tilde{g}_{i,K}(p_2).$ We denote this number by $K(p_1,p_2)$ to emphasize its dependence on the points $p_1,p_2.$ 

\begin{exa}\label{ord_cont}
Consider the same simplex $\check{\Delta}$ as in Example \ref{slice}. The projection $\tilde{C}$ of a complete intersection given by generic polynomials $f_1,f_2$ supported at $A,$ has two multiple points on the $1-$dimensional orbit of the toric variety $X_{P}$ corresponding to the edge $\Gamma=P^{(0,-1)}$. 
The polygon $P$ is a standard simplex of size $4,$ and since $\ind_v(A\cap\tilde{\Gamma})=2,$ each of the multiple points has $2$ preimages. Take any of them, and denote it by $p.$ Its preimages $p_1,p_2$ are not distinguished by $f_j^{\tilde{\Gamma}}$ while the polynomials $\tilde{g}_{j,1}$ do distinguish them. Therefore, we have $K(p_1,p_2)=1.$
\end{exa}

\begin{lemma}\label{ordcontact}
In the same notation as above, for a generic tuple of polynomials $f_1\ldots,f_{n+1}\in\C^A$ the order of contact between the projections of the branches of the curve $\tilde{\mathcal C}$ passing through the points $p_1,p_2,$ such that $\pi(p_1)=\pi(p_2)=p,$ equals to $K(p_1,p_2).$	
\end{lemma}
\begin{proof}
Let $K=K(p_1,p_2),$ where $p_1,p_2$ are as described above. Let us first compute the lower bound for the sought order of contact. 
\begin{utver}\label{lowerbound}
The order of contact between the projections of the branches of $\tilde{\mathcal C}$ passing through the points $p_1,p_2$ is greater or equal to $K.$ 
\end{utver}
\begin{proof}
Note that the points $p_1=(v_1,\ldots,v_n,u,0), p_2=(v'_1,\ldots,v'_n,u,0)$ under consideration are related in the following way. For some integers $a_1,\ldots,a_n,$ we have $v'_j=r_jv_j,$ where $r_j$ is some $(a_j)$-th root of unity. Also, note that rescaling the coordinate system in the same manner, i.e. performing the change of variables $\check{x}_i=r_jx_j,$ does not affect the coefficients monomials of $f_1,\ldots,f_{n+1}$ that do not distinguish the points $p_1$ and $p_2.$ For $1\leqslant j\leqslant n+1,$ by $F_j$ denote the part of $f_j$ which is invariant under 
this rescaling, and by $G_j$ the one that is not. 

Now, let us compute the order of contact at the point $p_1$ between the complete intersection curves $\tilde{\mathcal C}=\{f_1=\ldots=f_{n+1}=0\}$ and $X=\{F_1=\ldots=F_{n+1}=0\}.$ Let the curve $X$ be locally parametrized, i.e., in a neighborhood of the point $p_1,$ the curve $X$ is the image of a parametrization map $s\mapsto\varphi(s)=(\varphi_1(s),\ldots,\varphi_n(s),\varphi_{n+1}(s), \varphi_{n+2}(s)).$ Note also that since the $n+2$-th coordinate of $p_1$ is $0,$ the Taylor series of $\varphi_{n+2}$ at $p_1$ has no constant term, therefore it starts with the term of degree at least $1$ in $s.$ 

Substituting $\varphi(s)$ into the system $\{f_1=\ldots f_{n+1}=0\},$ we obtain the following system: $$\{G_1(\varphi(s))=\ldots=G_{n+1}(\varphi(s))=0\}.$$ The way the number $K=K(p_1,p_2)$ was defined implies that the polynomials $G_j$ are of the form $G_j=\sum_{m=K}^{\infty} t^m\tilde{g}_{i,m}(x_1\ldots,x_n,y).$ Thus, for $1\leqslant j\leqslant n+1,$ the leading term in $G_j(\varphi(s))$ is of degree $\geqslant K$ in $s.$ So, the order of contact between $X$ and $\tilde{\mathcal C}$ at $p_1$ is at least $K.$ The same is obviously true for the order of contact between $X$ and $\tilde{\mathcal C}$ at $p_2.$ 

Now we note that the change of variables $\check{x}_i=r_jx_j$ maps the branch of the curve $X$ passing through $p_1$ to the one passing through $p_2$ and at the same time does not change the defining polynomials of $X.$ Moreover, since this change of variables does not do anything with the last two of the $(n+2)$ coordinates, the projections of the two above-mentioned branches of $X$ coincide in some neighborhood of $p=\pi(p_1)=\pi(p_2).$ Finally, the projection $\pi$ does not decrease the order of contact between curves, which yields the desired inequality. 
\end{proof}

\begin{utver}\label{transverse}
In the same notation as above, suppose that $K(p_1,p_2)=1.$ Then for generic $f_1,\ldots,f_{n+1}\in\C^A,$ the projections of the branches of $\tilde{\mathcal C}$ that pass through $p_1$ and $p_2$ intersect transversally. 
\end{utver}

\begin{proof}
The idea is to compare the projections of the tangent lines to the curve $\tilde{C}$ at the points $p_1$ and $p_2$ and make sure that for almost all $f_1,\ldots,f_{n+1}\in\C^A,$ they do not coincide. 

To find the tangent lines at $p_1$ and $p_2$ we need to compute the kernel of the $(n+1)$--differential form $\bigwedge_{i=1}^{n+1} df_i$ at those points. Moreover, since we need to compare the projections of the tangent lines, the only components of $\bigwedge_{i=1}^{n+1} df_i$ that we have to look at are $dx_1\wedge\ldots\wedge dx_n\wedge dt$ and $dx_1\wedge\ldots\wedge dx_n\wedge dy.$ In the setting that is being considered, the corresponding coefficients will be the last two minors $M_{n+1}$ and $M_{n+2}$ of the $(n+2)\times(n+1)$--matrix $\mathcal J$ evaluated at the points $p_1$ and $p_2,$ where 

$$\mathcal J=\begin{pmatrix} 
\frac{\partial g_1}{\partial x_1}&\dots&\frac{\partial g_1}{\partial x_n}&\frac{\partial g_1}{\partial y}&\tilde{g}_{1,1}\\
\frac{\partial g_2}{\partial x_1}&\dots&\frac{\partial g_2}{\partial x_n}&\frac{\partial g_2}{\partial y}&\tilde{g}_{2,1}\\
\vdots & \ddots & \vdots & \vdots & \vdots \\
\frac{\partial g_{n+1}}{\partial x_1}&\dots&\frac{\partial g_{n+1}}{\partial x_n}&\frac{\partial g_{n+1}}{\partial y}&\tilde{g}_{n+1,1}\\
\end{pmatrix}.$$

Therefore, we need to show that for generic $f_1,\ldots,f_{n+1},$ we have 
\begin{equation}\label{jac}
\begin{vmatrix}
M_{n+1}(p_1)&M_{n+2}(p_1)\\
M_{n+1}(p_2)&M_{n+2}(p_2)\\
\end{vmatrix}\neq 0.
\end{equation}

The condition (\ref{jac}) is clearly algebraic, so, the set of tuples $f_1,\ldots,f_{n+1}\in\C^A$ satisfying it is Zariski open. To make sure it is everywhere dense, we need to show that it is non-empty. Indeed, suppose that for some $f_1,\ldots,f_{n+1}\in\C^A,$ we have  

\begin{equation}\label{tangentcond}
\begin{vmatrix}
M_{n+1}(p_1)&M_{n+2}(p_1)\\
M_{n+1}(p_2)&M_{n+2}(p_2)\\
\end{vmatrix}= 0.
\end{equation}

This is equivalent to the following system of equations:
\begin{equation}\label{syst}
\begin{cases}
M_{n+1}(p_1)=\lambda M_{n+2}(p_1)\\
M_{n+1}(p_2)=\lambda M_{n+2}(p_2)
\end{cases}
\end{equation}
for some $\lambda\in\C.$

The minors $M_{n+1}$ and $M_{n+2}$ are given by the following formulas:
\begin{equation}\label{Minors}
M_{n+1}=\begin{vmatrix} 
\frac{\partial g_1}{\partial x_1}&\dots&\frac{\partial g_1}{\partial x_n}&\tilde{g}_{1,1}\\
\frac{\partial g_2}{\partial x_1}&\dots&\frac{\partial g_2}{\partial x_n}&\tilde{g}_{2,1}\\
\vdots & \ddots & \vdots & \vdots \\
\frac{\partial g_{n+1}}{\partial x_1}&\dots&\frac{\partial g_{n+1}}{\partial x_n}&\tilde{g}_{n+1,1}\\
\end{vmatrix},
~M_{n+2}=\begin{vmatrix} 
\frac{\partial g_1}{\partial x_1}&\dots&\frac{\partial g_1}{\partial x_n}&\frac{\partial g_1}{\partial y}\\
\frac{\partial g_2}{\partial x_1}&\dots&\frac{\partial g_2}{\partial x_n}&\frac{\partial g_2}{\partial y}\\
\vdots & \ddots & \vdots & \vdots \\
\frac{\partial g_{n+1}}{\partial x_1}&\dots&\frac{\partial g_{n+1}}{\partial x_n}&\frac{\partial g_{n+1}}{\partial y}\\
\end{vmatrix}.
\end{equation}
The matrices in (\ref{Minors}) are almost identical except for the last column, let us expand their determinants along this column. Then, we obtain:
\begin{align}\label{determexpand}
M_{n+1}=\sum_{j=1}^{n+1}(-1)^{n+1+j}\tilde{g}_{j,1}D_j=(-1)^{n+2}\tilde{g}_{1,1}D_1+\sum_{j=2}^{n+1}(-1)^{n+1+j}\tilde{g}_{j,1}D_j;\\
M_{n+2}=\sum_{j=1}^{n+1}(-1)^{n+1+j}\frac{\partial g_j}{\partial y}D_j=(-1)^{n+2}\frac{\partial g_1}{\partial y}D_1+\sum_{j=2}^{n+1}(-1)^{n+1+j}\frac{\partial g_j}{\partial y}D_j, 
\end{align}
where $D_j$ are the cofactors of the entries in the last columns. 
We have that for generic $f_1,\ldots,f_{n+1},$ the intersection points of the curve $\tilde{\mathcal{C}}$ with the orbit of $X_{\Delta}$ corresponding to the facet $\tilde{\Gamma}$ are non-singular, therefore the minor $M_{n+1}$ does not vanish at $p.$ Thus, at least one of the cofactors $D_j$ is not zero. Without loss of generality, let us assume that the cofactor $D_1$ is not zero. Substituting equations in (\ref{determexpand}) into the system (\ref{syst}), we obtain:
\begin{equation}\label{syst1}
\begin{cases}
(-1)^{n+2}(\tilde{g}_{1,1}-\lambda\frac{\partial g_1}{\partial y})D_1\mid_{p_1}=\sum_{j=2}^{n+1}(-1)^{n+j}(\tilde{g}_{j,1}-\lambda\frac{\partial g_j}{\partial y})D_j\mid_{p_1}\\
(-1)^{n+2}(\tilde{g}_{1,1}-\lambda\frac{\partial g_1}{\partial y})D_1\mid_{p_2}=\sum_{j=2}^{n+1}(-1)^{n+j}(\tilde{g}_{j,1}-\lambda\frac{\partial g_j}{\partial y})D_j\mid_{p_2}.
\end{cases}
\end{equation}
Now let us note, that if the polynomial $\tilde{g}_{1,1}$ distinguishes the points $p_1$ and $p_2,$ then so does at least one of its monomials. If we change the coefficient of this monomial in $\tilde{g}_{1,1},$ then the left-hand sides of each of the equalities in (\ref{syst1}) change independently, while the right-hand sides remain unchanged. Therefore the equations become no longer true, so the determinant in (\ref{tangentcond}) does not vanish anymore. 
\end{proof}

It remains to show that Proposition \ref{transverse} together with Proposition \ref{lowerbound} implies Lemma $\ref{ordcontact}$ for $K\neq 1.$ Let us note that if all the polynomials $\tilde{g}_{i,m}$ for all $1\leqslant i\leqslant n+1$ and all  $m$ except for $m=0$ and $m=K$ are zero, then the proof is the same as in case $K=1,$ up to a change of variable $\check{t}=t^{K}.$  Therefore, we proved that for almost all $f_1,\ldots, f_{n+1}$ contained in the plane in $\C^A$ given by equations $c_{\alpha}=0,$ where $\alpha\notin\supp(f_i^{\tilde{\Gamma}})\cup\supp(\tilde{g}_{i,K}),$ the order of contact of the projections of the branches passing through $p_1,p_2$ is equal to $K=K(p_1,p_2).$ 

Finally, the intersection index is upper semi-continuous. Therefore, perturbing the coefficients of monomials in $\alpha\in \supp(\tilde{g}_{i,m}),~1\leqslant m\leqslant K-1,$ does not increase the sought order of contact, therefore, it remains smaller or equal to $K.$ At the same time, Proposition \ref{lowerbound} gives the lower bound for the sought order of contact, which is also equal to $K.$ Thus Lemma \ref{ordcontact} is proved.
\end{proof}

\subsection{Singular Points of the Curve $\mathcal C$ at Infinity}\label{s2}
Let $\Gamma\subset\Delta$ be a facet contained in a hyperplane given by a linear equation of the form $\ell(e_1,\ldots,e_{n+2})=c.$

We now construct a sequence of integers $i^{\Gamma}=(i_1^{\Gamma},i_2^{\Gamma},\ldots)$ as follows. 

Set $B_1^{\Gamma}=A\cap\Gamma.$ For every $r>1,$ we define  $$B_r^{\Gamma}=B_{r-1}^{\Gamma}\cup(A\cap\{\ell(e_1,\ldots,e_{n+2})=c\pm (r-1)\}),$$ depending on the way $\Delta$ is positioned relative to the hyperplane containing $\Gamma.$ 
Finally, for every $r\geqslant 1,$ we set $$i_r^{\Gamma}=\ind_{v}(B_r^{\Gamma}).$$

\begin{rem}\label{stabind}
It is clear that for every $r,$ the element $i_r^{\Gamma}$ divides $i_{r-1}^{\Gamma}.$ Moreover, since for the set $A$ we have $\ind_v(A)=1,$ any such sequence stabilizes to $1$.
\end{rem}
In the same notation as above, we state the following result. 
\begin{lemma} \label{fpsinfinity}
For a generic tuple of polynomials $f_1,\ldots,f_{n+1}\in\C^A,$ all the singularities of the curve $\mathcal C$ at infinity are $i^{\Gamma}-$forking paths singularities for non-horizontal facets $\Gamma\subset\Delta.$ 
\end{lemma}
\begin{proof}
This statement is a straightforward corollary of Lemma \ref{inf_br}, Lemma \ref{ordcontact} and Definition \ref{def_fps}. 
\end{proof}

\subsection{Comparing the Euler characteristics of the Complete Intersection and Its Projection}\label{main}

First, let us make the following important observation. 
\begin{utver}
If $A\subset\Z^{n+2}$ satisfies Assumption \ref{indall}, then for generic $f_1,\ldots,f_{n+1}\in\C^A,$ a generic point $p\in\mathcal{C}$ has exactly one preimage. Moreover, the set of points in $\mathcal{C}$ that have more than $1$ preimage is finite.
\end{utver}

\begin{proof}
The first of the statements follows from Assumption \ref{indall} together with Theorem \ref{num_fib} applied to the essential collection $\mathcal A=\Big(\underbrace{A,\ldots,A}_{n+1\text{~times}},\supp(y-y_0),\supp(t-t_0)\Big).$
Indeed, in the notation of Theorem \ref{num_fib}, we have $I=n+3,~L=L'=\Z^{n+2}$ and $k=0.$ So, $d(\mathcal A)=1,$ which proves the first part of the Proposition.
	
Moreover, note that since the projection $\pi\mid_{\tilde{\mathcal C}}$ is a regular morphism of varieties, the set of points having more than one preimage is a constructible set, which we denote by $\mathcal B.$ Let us show that it is finite.
	
Indeed, assume that $\mathcal B$ is infinite. Then it is of positive dimension, therefore, it contains an algebraic curve (possibly with finitely many punctured points). An algebraic curve in $\CC^2$ has at least one branch at infinity. Lemma \ref{fpsinfinity} implies that this branch is one of the branches of a forking paths singularity of the curve $\mathcal C$ at infinity. At the same time we note that the preimage of this branch contains a pair of branches of the curve $\tilde{\mathcal{C}}.$ However, Lemma \ref{ordcontact}, Lemma \ref{fpsinfinity} and Remark \ref{stabind} imply that the projections of any two branches of the curve $\tilde{\mathcal{C}}$ at infinity cannot coincide. 
\end{proof}

\begin{rem}\label{prim_rem}
Note that it suffices to prove Theorem \ref{lemmain} for a support set $A$ whose convex hull $\Delta$ satisfies Assumption \ref{primitive}. Indeed, suppose $\Delta$ does not satisfy this assumption. Then, using a monomial change of variables of the form $\check{x}_i=x_i,\check{y}^{N!}=y,\check{t}^{N!}$ for $N$ big enough, we obtain the polytope $\check{\Delta}$ that satisfies the desired assumption. Moreover, as was discussed in Remark \ref{primitive_exp}, under this change of variables, the left-hand side of (\ref{mainformula1}) is multiplied by $N!^2.$ Note that each of the summands in the right-hand side of (\ref{mainformula1}) is also multiplied by $N!^2.$ Therefore, the $N!^2$ factor can be cancelled out. So, Theorem \ref{lemmain} holds true independently of Assumption \ref{primitive}.
\end{rem}

\begin{proof}[Proof of Theorem \ref{lemmain}]
Suppose that the polytope $\Delta$ satisfies Assumption \ref{primitive}. 
The number $\mathcal{D}$ can be deduced from the following system of equations: 
	\begin{equation}\label{Dplus}
	\begin{cases}
	\chi(\tilde{\mathcal C})=-(n+1)\vol(\Delta)\\
	\chi(\tilde{\mathcal C})+\sum\limits_{\Gamma\in\mathcal H(\Delta)}\vol(\Gamma)=\chi(\mathcal C)+\sum\limits_{s\in\mathcal{S}}(b(s)-1)\\
	\chi(\mathcal C)=-\area(P)-\sum\limits_{s\in\mathcal{S}}(b(s)-1)+2\mathcal{D}+\sum_{\Gamma\in\mathcal{F}\Delta\setminus\mathcal H(\Delta)}\vol(\Gamma)\sum\limits_1^{\infty}(i_r^{\Gamma}-1),
	\end{cases}
	\end{equation}
	where $b(s)$ is the number of branches passing through the singular point $s\in\mathcal{S}.$
	
	The first equation is a straightforward corollary of Theorem \ref{chicomplint} (see Example \ref{chici} for details).
	
	The second equation follows from the additivity property of Euler characteristic combined with  Bernstein--Kouchnirenko theorem: to obtain the Euler characteristic of $\tilde{\mathcal{C}}$ from the one of $\mathcal{C}$ one just needs to count the multiple points as many times, as is the number of branches passing through them and subtract the number of points coming from vertical asymptotes of the curve $\tilde{\mathcal C}$ (and thus contained in the set $\mathcal C\setminus\pi(\tilde{\mathcal C})$).
	
	Finally, the third equation can be deduced as follows. Let $Y$ be a smooth curve having the same Newton polygon $P$ as the curve $\mathcal C.$ Theorem \ref{chicomplint} implies that its Euler characteristic is equal to $-\area(P)$ (see Example \ref{chiproj}).
	Now, we compare the Euler characteristic of $Y$ with the one of $\mathcal{C},$ and we have the following equality:
	\begin{equation}\label{euler_comp}
	\chi(Y)=\chi(\mathcal{C})-\sum\limits_{s\in\mathcal{S}}1+\sum\limits_{s\in\mathcal{S}}(b(s)-2\delta(s))\\+\sum\limits_{S\in \mathrm{FPS}}\chi([\mathrm{Milnor~fiber~of~}S]\cap\CC^2),
	\end{equation}
	where the third summation runs over all forking paths singularities of $\mathcal{C}$ at infinity.
	
Indeed, the right-hand side of (\ref{euler_comp}) can be interpreted as follows. We puncture small neighborhoods of singular points of the curve $\mathcal{C}.$ Then we replace those neighborhoods with their Milnor fibers. The latter are of Euler characteristic $b(s)-2\delta(s),$ by Milnor formula (see Theorem 10.5 in \cite{Milnor}). Finally, we add the Milnor fibers of the forking paths singularities of $\mathcal{C}.$ Then, the Euler characteristic of the curve that we obtain this way equals to the Euler characteristic of the smooth curve $Y.$ 

Remark \ref{prim_rem} implies that the desired statement holds true without Assumption \ref{primitive}, which concludes the proof of Theorem \ref{lemmain}.
\end{proof}

\section{Further Discussion}\label{fd}

\subsection{On the Singular Points of the curve $\mathcal C$} 
We begin this subsection by stating the following conjecture.

\begin{conj}\label{c1}
	Let $A\in\Z^{n+2}$ be a finite set of full dimension containing $0\in\Z^{n+2}$ and satisfying Assumption \ref{indall}. For generic polynomials  $f_1,\ldots,f_{n+1}\in\C^{A},$ the image of the complete intersection curve  ${\tilde{\mathcal{C}}=\{f_1=\ldots=f_{n+1}=0\}}\subset \CC^{n+2}$ under the projection $\pi\colon\CC^{n+2}\to\CC^2$ forgetting the first $n$ coordinates is an algebraic plane curve $\pi(\tilde{\mathcal C})$ with punctured points, and its singular locus consists only of nodes.
\end{conj}

\begin{rem}
	In other words, the statement of Conjecture \ref{s1} means that the singular points of the curve $\mathcal C$ which are not nodes are contained in the set $\mathcal C\setminus\pi(\tilde{\mathcal C}).$ 
\end{rem}

Conjecture \ref{c1} naturally leads to the following question. 

\begin{prb}\label{p2}
	In the same notation as above, compute the sum of the $\delta$-invariants of the singularities of the curve $\mathcal C$ at the punctured points of the curve $\pi(\tilde{\mathcal C}).$ 
\end{prb}

Given a solution to Problem \ref{p2}, one could then compute the number of nodes of the curve $\pi(\tilde{\mathcal C})$ itself. Under the following assumption, we state a conjecture for this number.

\begin{predpol}\label{horiz_latt}
For every horizontal facet $\Gamma\in\mathcal H(\Delta),$ the image of $\Gamma\cap A$ under the projection $\rho\colon\Z^{(n+2)}\twoheadrightarrow\bigslant{\Z^{(n+2)}}{\langle e_{n+1},e_{n+2}\rangle}$ generates a saturated sublattice. 
\end{predpol}

\begin{rem}
Assumption \ref{horiz_latt} guarantees that for generic $f_1,\ldots,f_{n+1}\in\C^{A},$ each of the singularities at the points in $\mathcal C\setminus\pi(\tilde{\mathcal C})$ has only one irreducible component. For $n=1,$ this assumption is satisfied for any support set $A,$ while for $n>1$ it is non-trivial.
\end{rem}

To state the conjecture below, we need to extend the construction of the sequences $i_r^{\Gamma}$ (see Subsection \ref{s2}) to horizontal facets of the polytope $\Delta.$ 

If we follow the same procedure as in Subsection \ref{s2}, we will get a sequence $j_k^{\Gamma}$ of indices of the form $(\underbrace{\infty,\infty,\ldots,\infty}_{ R \text{ times}},i^{\Gamma}_{R+1},i^{\Gamma}_{R+2},\ldots),$ where $R$ is the smallest strictly positive number such that  $B_{R+1}^{\Gamma}\neq A\cap\Gamma.$ 
Finally, we set $$i_r^{\Gamma}=j_{R+r}^{\Gamma}.$$

\begin{conj}\label{c2}
	Let $A\subset\Z^{n+2},~n\geqslant 1,$ be a finite set of full dimension containing $0$ and satisfying Assumptions \ref{indall} and \ref{horiz_latt}. Let $\Delta\subset\R^{n+2}$ be its convex hull. For generic $f_1,\ldots,f_{n+1}\in\C^A,$ the image of the $1-$dimensional complete intersection $\tilde{\mathcal{C}}=\{f_1=\ldots=f_{n+1}=0\}$ under the projection $\pi\colon\CC^{n+2}\to\CC^2$ forgetting the first $n$ coordinates, is an algebraic plane curve $\pi(\tilde{\mathcal C})$ with punctured points.  Its singular locus $\mathcal S$ (not including the punctured points) consists of isolated singular points, and the number $\mathcal D=\sum\limits_{s\in\mathcal S}\delta(s)$ can be computed via the following formula:  
	\begin{equation}\label{conjform}
	\mathcal D=\dfrac{1}{2}\Bigg(\area(P)-(n+1)\vol(\Delta)+\sum\limits_{\Gamma\in\mathcal{H}(\Delta)}\vol(\Gamma)\big(2i_1^{\Gamma}-(i_1^{\Gamma})^2\big)-\sum\limits_{\Gamma\in\mathcal{F}(\Delta)}\vol(\Gamma)\sum\limits_1^{\infty}(i_r^{\Gamma}-1)\Bigg),
	\end{equation}
	where $\delta(s)$ is the $\delta$-invariant of the singular point $s$, the polygon $P=\int_{\pi}(\Delta)$ is the Minkowski integral of $\Delta$ with respect to the projection $\pi$, the set $\mathcal{F}(\Delta)$ is the set of all the facets of the polytope $\Delta,$ and  $\mathcal{H}(\Delta)\subset\mathcal{F}(\Delta)$ is the set of all the horizontal facets of the polytope $\Delta.$ 
\end{conj}

We will now prove Conjecture \ref{c2} under an additional assumption on the support set $A$. We begin with a result which is a straightforward corollary of Theorem \ref{lemmain}.

\begin{theor}\label{developed}
	Let $A\subset\Z^{n+2}$ be a finite set of full dimension, satisfying Assumption \ref{indall}, and let $\Delta\subset\R^{n+2}$ be its convex hull. Let $\pi\colon\CC^{n+2}\to\CC^2$ be the projection forgetting the first $n$ coordinates. 
	In the same notation as above, suppose that the set $\mathcal H(\Delta)$ is empty. 
	Then for generic $f_1,\ldots,f_{n+1}\in\C^A,$ the image of the curve $\tilde{\mathcal{C}}=\{f_1=\ldots=f_{n+1}=0\}$ under the projection $\pi$ is an algebraic plane curve $\mathcal{C}$ an  algebraic  plane  curve $\mathcal C$ (with no points punctured), whose singular locus $\mathcal S$ consists of isolated singular points. The number $\mathcal D=\sum\limits_{s\in\mathcal S}\delta(s)$ can be computed via the following formula:  
	\begin{equation}
	\mathcal D=\dfrac{1}{2}\Bigg(\area(P)-(n+1)\vol(\Delta)-\sum\limits_{\Gamma\in\mathcal{F}(\Delta)}\vol(\Gamma)\sum\limits_1^{\infty}(i_r^{\Gamma}-1)\Bigg),
	\end{equation}
	where $\delta(s)$ is the $\delta$-invariant of the singular point $s$, $P=\int_{\pi}(\Delta),$ the set $\mathcal{F}(\Delta)$ is the set of all facets of the polytope $\Delta.$
\end{theor}

Theorem \ref{developed} applies to the case when the polytope $\Delta$ is developed with respect to the plane of projection $\pi.$ We can, in fact, prove a stronger statement, namely, Theorem \ref{maintheor}, which works for a slightly more general case when the support set $A$ satisfies an additional property (see Assumption \ref{horizontal} below). 

\begin{predpol}\label{horizontal}
	For every horizontal face $\tilde{\Gamma}\subset \Delta,$ one of the following two possibilities is realized:
	
	\begin{itemize}
		\item the preimage $\tilde{\Gamma}\subset\Delta$ is not a facet;
		\item $\tilde{\Gamma}\subset\Delta$ is a facet, and the set $A\setminus\tilde{\Gamma}$ is at lattice distance $1$ from $\tilde{\Gamma}$.
	\end{itemize}
\end{predpol}

\begin{rem}
	One can rephrase the second condition in Assumption \ref{horizontal} as follows: if $\tilde{\Gamma}\subset\Delta$ 
	is a facet contained in a hyperplane given by the equation $\ell(e_1,\ldots,e_{n+2})=c,$ then the set $A\cap\{\ell(e_1,\ldots,e_{n+2})=c\pm 1\}\subset\Z^{n+2}$ is non-empty. 
\end{rem}

The set $\mathcal{C}\setminus\pi(\tilde{\mathcal{C}})$ consists of the projections of the intersection points of the curve $\tilde{\mathcal C}$ with the orbits of the toric variety $X_{\Delta}$ corresponding to the horizontal facets of the polytope $\Delta.$ As we will see further, under Assumptions \ref{horiz_latt} and \ref{horizontal}, for generic $f_1,\ldots,f_{n+1}\in\C^A,$ the points in $\mathcal{C}\setminus\pi(\tilde{\mathcal{C}})$ are smooth, therefore each of those points contributes exaclty $1$ to the Euler characteristic of the curve $\mathcal C.$

\begin{lemma}\label{punct}
Let $A\subset\Z^{n+2}$ be a finite set of full dimension, satisfying Assumptions \ref{indall},\ref{horiz_latt} and \ref{horizontal}. Then for generic $f_1,\ldots,f_{n+1}\in\C^A,$ the following statements are true: 
	\begin{enumerate}
		\item let $\Gamma\subset\Delta$ be a horizontal facet and $p$ be an intersection point of the closure of the curve $\tilde{\mathcal C}$ with the corresponding orbit $\mathcal O_{\Gamma}\subset X_{\Delta}.$ The point $\pi(p)$ has exactly one preimage;
		\item Let $p$ be as described in part 1. Then the projection of the tangent line at $p$ is $1$-dimensional.
	\end{enumerate}
\end{lemma}
\begin{proof}
	
	{\bf Part 1.} Let $\Gamma\subset\Delta$ be a horizontal facet. Up to a monomial change of variables one can assume the truncations $f_1^{\Gamma},\ldots,f_{n+1}^{\Gamma}$ to be polynomials in $(x_2,\ldots,x_n,y,t).$ Thus the polynomials $f_1,\ldots,f_{n+1}$ can be written in the following form: $f_i=g_i(x_2,\ldots,x_n,y,t)+\sum_{m=1}^{\infty}x_1^m\tilde{g}_{i,m}(x_2,\ldots,x_n,y,t),$ where $g_i=f_i^{\Gamma}.$ Assumption \ref{horizontal} implies that the polynomials $g_{i,1}$ are nonzero for all $1\leqslant i\leqslant n+1.$ Let $q=(u_2,\ldots,u_n,y_0,t_0)$ be a root for the system $f_1^{\Gamma}=\ldots=f_{n+1}^{\Gamma}=0$ and $p=(0,u_2,\ldots,u_n,y_0,t_0)\in\mathcal{O}_{\Gamma}$ be the corresponding limiting point of the curve $\tilde{\mathcal{C}}.$ 
	
	We need to show that there is no other point $p'$ in the closure of the curve $\tilde{ \mathcal {C}}\subset X_{\Delta}$ such that $\pi(p')=\pi(p)=(y_0,t_0).$ Suppose that such a point $p'$ exists. Then Proposition \ref{diff_facets} implies that the point $p'$ either belongs to the same orbit $\mathcal{O}_{\Gamma}$ as the point $p,$ or is in the torus $\CC^{n+2}.$ 
	
	As we will see further, due to Assumption \ref{horiz_latt}, the first of the two options is not the case. In other words, we will show that for any two distinct roots $q$ and $q'$ for the truncated system $f_1^{\Gamma}=\ldots=f_{n+1}^{\Gamma}=0,$ at least one of the characters $y^1$ or $t^1$ attains distinct values on $q$ and $q'.$  
	
    The latter is equivalent to the following statement: if the overdetermined system $$\{f_1^{\Gamma}=\ldots=f_{n+1}^{\Gamma}=y-y_0=t-t_0=0\}$$ is consistent, then it has a unique root.   
	
	Let $\Lambda_{\Gamma}$ be the lattice in $\Lambda=\bigslant{Z^{n+2}}{\langle e_{n+1},e_{n+2}\rangle}\simeq \Z^n$ generated by the image of $\Gamma\cap A$ under the projection $\rho\colon\Z^{(n+2)}\twoheadrightarrow\bigslant{\Z^{(n+2)}}{\langle e_{n+1},e_{n+2}\rangle}.$ Note that its dimension is equal to $n-1.$ Assumption \ref{horiz_latt} implies that the lattice $\Lambda_{\Gamma}\subset\Lambda$ is saturated, therefore the sets $(A\cap\Gamma),\supp(y-y_0),\supp(t-t_0)\subset\Z^{n+2}.$ generate a saturated dimension $n+1$ sublattice $L\simeq\Z^{n+1}$ in $\Z^{n+2}.$ Moreover, these sets cannot be shifted to the same proper sublattice in $L\simeq\Z^{n+1}.$ 
	
	The statement that we are proving is then a special case of Theorem \ref{num_fib}: we take the essential collection $\mathcal A$ consisting of $n+1$ copies of the set $A\cap\Gamma\subset\Z^{n+2}$ and two more sets $\supp(y-y_0),\supp(t-t_0)\subset\Z^{n+2}.$
	In the notation of this theorem, the formula for the number $d(\mathcal A)$ in our special case gives the answer $1,$ which means that for generic $f_1,\ldots,f_{n+1}\in\C^A,$ the system $\{f_1^{\Gamma}=\ldots=f_{n+1}^{\Gamma}=y-y_0=t-t_0=0\}$ has exactly one root $q.$ Indeed, in our case $L'=L,$ because $L$ is saturated, and $k=0,$ because the collection is essential.

	Therefore, any other point $p'\in\tilde{C}$ such that $\pi(p')=\pi(p)$ should belong to $\CC^{n+2},$ and thus any monomial of degree at least $1$ in $x_1$ distinguishes the points $p$ and $p'$. Assumption \ref{horizontal} implies the existence of such a monomial. Thus we can construct a polynomial $\tilde{f}$ such that $\tilde{f}(p)=0,$ and $\tilde{f}(p')\neq 0.$ Adding $\varepsilon\cdot\tilde{f}$ to one of the $f_i$'s does not affect the point $p$ in any way, while $p'$ does not belong to the corresponding complete intersection anymore. At the same time, the polynomial $(f_i+\varepsilon\cdot\tilde{f})\mid_p:=f_i(x_1,u_2,\ldots,u_n,y_0,t_0)+\varepsilon\cdot\tilde{f}(x_1,u_2,\ldots,u_n,y_0,t_0)$ is a polynomial in one variable $x_1,$ which has only finitely many roots $r_1,\ldots,r_m.$ Therefore, for generic $f_1,\ldots,f_{n+1},$ the restrictions $f_j\mid_p,~j\neq i$ do not all vanish at $r_1,\ldots,r_m.$ This concludes the proof of Part 1. 
	
	{\bf Part 2.} In the same notation as above, to describe the tangent line to the point $p,$ one needs first to compute the kernel of the differential form $\bigwedge_{i=1}^{n+1}df_i$ at the point $p.$ Since we are interested in the projection of the tangent line, the only components that we need to compute are $dx_1\wedge\ldots\wedge dx_{n}\wedge dy$ and $dx_1\wedge\ldots\wedge dx_{n}\wedge dt.$
	The corresponding coefficients are the last two minors $M_{n+1}$ and $M_{n+2}$ of the $(n+2)\times(n+1)$--matrix $\mathcal J$ evaluated at $p,$ where 
	
	$$\mathcal J=\begin{pmatrix} 
	\tilde{g}_{1,1}&\frac{\partial g_1}{\partial x_2}&\dots&\frac{\partial g_1}{\partial x_n}&\frac{\partial g_1}{\partial y}&\frac{\partial g_1}{\partial t}\\
	\tilde{g}_{2,1}&\frac{\partial g_2}{\partial x_2}&\dots&\frac{\partial g_2}{\partial x_n}&\frac{\partial g_2}{\partial y}&\frac{\partial g_2}{\partial t}\\
	\vdots & \vdots & \ddots & \vdots & \vdots & \vdots\\
	\tilde{g}_{n+1,1}&\frac{\partial g_{n+1}}{\partial x_2}&\dots&\frac{\partial g_{n+1}}{\partial x_n}&\frac{\partial g_{n+1}}{\partial y}&\frac{\partial g_{n+1}}{\partial t}\\
	\end{pmatrix}.$$
	
	The projection of the tangent line to the curve $\tilde{\mathcal C}$ at the point $p$ is $1$-dimensional if and only if at least one of $M_{n+1}(p)$ and $M_{n+2}(p)$ is not zero. This condition is algebraic, therefore, the set of all tuples $f_1,\ldots,f_{n+1}\in\C^A,$ that satisfy it is Zariski open. Thus, the only thing that we still need to show is that this set is non-empty. Indeed, let $M_{n+1}(p)$ and $M_{n+2}(p)$ both vanish for some $f_1,\ldots,f_{n+1}\in\C^A.$ The latter means that the following equalities hold: 
	\begin{align*}
	M_{n+1}=\begin{vmatrix} 
	\tilde{g}_{1,1}&\frac{\partial g_1}{\partial x_2}&\dots&\frac{\partial g_1}{\partial x_n}&\frac{\partial g_1}{\partial t}\\
	\tilde{g}_{2,1}&\frac{\partial g_2}{\partial x_2}&\dots&\frac{\partial g_2}{\partial x_n}&\frac{\partial g_2}{\partial t}\\
	\vdots & \vdots & \ddots & \vdots & \vdots\\
	\tilde{g}_{n+1,1}&\frac{\partial g_{n+1}}{\partial x_2}&\dots&\frac{\partial g_{n+1}}{\partial x_n}&\frac{\partial g_{n+1}}{\partial t}
	\end{vmatrix}=0,
	~M_{n+2}=\begin{vmatrix} 
	\tilde{g}_{1,1}&\frac{\partial g_1}{\partial x_2}&\dots&\frac{\partial g_1}{\partial x_n}&\frac{\partial g_1}{\partial y}\\
	\tilde{g}_{2,1}&\frac{\partial g_2}{\partial x_2}&\dots&\frac{\partial g_2}{\partial x_n}&\frac{\partial g_2}{\partial y}\\
	\vdots & \vdots & \ddots & \vdots & \vdots  \\
	\tilde{g}_{n+1,1}&\frac{\partial g_{n+1}}{\partial x_2}&\dots&\frac{\partial g_{n+1}}{\partial x_n}&\frac{\partial g_{n+1}}{\partial y}\\
	\end{vmatrix}=0.
	\end{align*}
	Let us look at the first equality. We have that the Jacobian of $g_1,\ldots,g_{n+1}$ does not vanish at $p.$ Expanding it along the second to last column, we conclude that at least one of the cofactors is not zero. Without loss of generality, assume that this is the cofactor $D_1$ of $\frac{\partial g_1} {\partial y}.$ Now, let us expand the determinant $M_{n+1}$ along the first column. The cofactors $D_j$ of the elements $\tilde{g}_{j,1}$ are exactly the same as the ones of $\frac{\partial g_j} {\partial y}$ in the Jacobian matrix of $g_1,\ldots,g_{n+1}.$ Thus we have: 
	$$M_{n+1}\mid_{p}=\tilde{g}_{1,1}D_1\mid_{p}+\sum_{j=2}^{n+1}(-1)^{n+1+j}\tilde{g}_{j,1}D_j\mid_{p}=0,$$ or, equivalently,
	\begin{equation}\label{nondeg}
	\tilde{g}_{1,1}D_1\mid_{p}=\sum_{j=2}^{n+1}(-1)^{n+j}\tilde{g}_{j,1}D_j\mid_{p}.
	\end{equation}
	
	By Assumption \ref{horizontal}, the polynomial $\tilde{g}_{1,1}$ is nonzero, therefore, it has at least one monomial with a nonzero coefficient. If we change this coefficient, the left-hand side of (\ref{nondeg}) changes, while the right-hand side does not. Therefore, the equality is no longer true, which concludes the proof of Part $2.$
\end{proof}

The following result is a straightforward corollary of Theorem \ref{lemmain} and Lemma \ref{punct}. 

\begin{theor}\label{maintheor}
	Let $A\subset\Z^{n+2},~n\geqslant 1,$ be a finite set of full dimension containing $0$ and satisfying Assumptions \ref{indall}, \ref{horiz_latt} and \ref{horizontal}. Let $\Delta\subset\R^{n+2}$ be its convex hull. For generic $f_1,\ldots,f_{n+1}\in\C^A,$ the image of the $1-$dimensional complete intersection ${\tilde{\mathcal{C}}=\{f_1=\ldots=f_{n+1}=0\}}$ under the projection $\pi\colon\CC^{n+2}\to\CC^2$ forgetting the first $n$ coordinates is an algebraic plane curve with punctured points. Its singular locus $\mathcal S$ (not including the punctured points) consists of isolated singular points, and the number $\mathcal D=\sum\limits_{s\in\mathcal S}\delta(s)$ can be computed via the following formula:  
	\begin{equation}\label{mainformula}
	\mathcal D=\dfrac{1}{2}\Bigg(\area(P)-(n+1)\vol(\Delta)+\sum\limits_{\Gamma\in\mathcal{H}(\Delta)}\vol(\Gamma)-\sum\limits_{\Gamma\in\mathcal{F}(\Delta)}\vol(\Gamma)\sum\limits_1^{\infty}(i_r^{\Gamma}-1)\Bigg),
	\end{equation}
	where $\delta(s)$ is the $\delta$-invariant of the singular point $s$, the polygon $P=\int_{\pi}(\Delta)$ is the Minkowski integral of $\Delta$ with respect to the projection $\pi$, and the set $\mathcal{F}(\Delta)$ is the set of all facets of the polytope $\Delta.$
\end{theor}

\begin{proof}
Assumption \ref{horizontal} guarantees that for every horizontal facet $\Gamma\subset\Delta,$ we have $i^{\Gamma}=(1,1,\ldots).$ Therefore, in this case, we have the equality  $$\sum\limits_{\Gamma\in\mathcal{F}(\Delta)}\vol(\Gamma)\sum(i_r^{\Gamma}-1)=\sum\limits_{\Gamma\in\mathcal{F}(\Delta)\setminus\mathcal{H}(\Delta)}\vol(\Gamma)\sum(i_r^{\Gamma}-1).$$
Theorem \ref{lemmain} gives the formula for the sum $\mathcal  D$ of the $\delta$-invariants of all the singularities of the closure $\mathcal C$ of the curve $\pi(\tilde{\mathcal C}).$ Lemma \ref{punct} implies that all the points of the complement $\mathcal C\setminus \pi(\tilde{\mathcal C})$ are smooth, therefore those points do not contribute to the sum $D,$ which concludes the proof of the theorem. 
\end{proof}

\begin{rem}
Note that Theorem \ref{maintheor} agrees with Conjecture \ref{c2}. Indeed, since for every horizontal facet $\Gamma\subset\Delta,$ we have $i^{\Gamma}=(1,1,\ldots).$ So, we have  $$\sum\limits_{\Gamma\in\mathcal{H}(\Delta)}\vol(\Gamma)\big(2i_1^{\Gamma}-(i_1^{\Gamma})^2\big)=\sum\limits_{\Gamma\in\mathcal{H}(\Delta)}\vol(\Gamma)(2-1^2)=\sum\limits_{\Gamma\in\mathcal{H}(\Delta)}\vol(\Gamma).$$
All the other summands in (\ref{conjform}) are the same as in (\ref{mainformula}), so these two expressions give the same answer.
\end{rem}

\begin{exa}
	Let us apply Theorem \ref{maintheor} to the following two support sets (see Figure 7): 
	\begin{align*}
	&A_1=\{(0,0,0),(1,0,0),(2,0,0),(3,0,0),(0,1,0),(0,0,1)\}\\ 
	&A_2=\{(0,0,0),(1,0,0),(3,0,0),(0,1,0),(0,0,1)\}
	\end{align*} 
	\begin{center}
		\begin{tikzpicture}
		\pattern[pattern=north east lines, pattern color=orange] (0,7)--(2,1)--(1,0)--(0,7);
		\pattern[pattern=north east lines, pattern color=orange] (6,7)--(8,1)--(7,0)--(6,7);
		\draw[ultra thick] (0,1)--(0,7);
		\draw[ultra thick] (0,1)--(1,0);
		\draw[ultra thick] (0,7)--(2,1);
		\draw[ultra thick] (1,0)--(2,1);
		\draw[ultra thick] (0,7)--(1,0);
		\draw[dashed, ultra thick] (0,1)--(2,1);
		\draw[fill,red] (0,1) circle [radius=0.1];
		\draw[fill,red] (0,7) circle [radius=0.1];
		\draw[fill,red] (1,0) circle [radius=0.1];
		\draw[fill,red] (2,1) circle [radius=0.1];
		\draw[fill,red] (0,3) circle [radius=0.1];
		\draw[fill,red] (0,5) circle [radius=0.1];
		\draw[ultra thick] (6,1)--(6,7);
		\draw[ultra thick] (6,1)--(7,0);
		\draw[ultra thick] (6,7)--(8,1);
		\draw[ultra thick] (7,0)--(8,1);
		\draw[ultra thick] (6,7)--(7,0);
		\draw[dashed, ultra thick] (6,1)--(8,1);
		\draw[fill,red] (6,1) circle [radius=0.1];
		\draw[fill,red] (7,0) circle [radius=0.1];
		\draw[fill,red] (8,1) circle [radius=0.1];
		\draw[fill,red] (6,7) circle [radius=0.1];
		\draw[fill,red] (6,3) circle [radius=0.1];
		\draw [ultra thick, fill=white] (6,5) circle [radius=0.1];
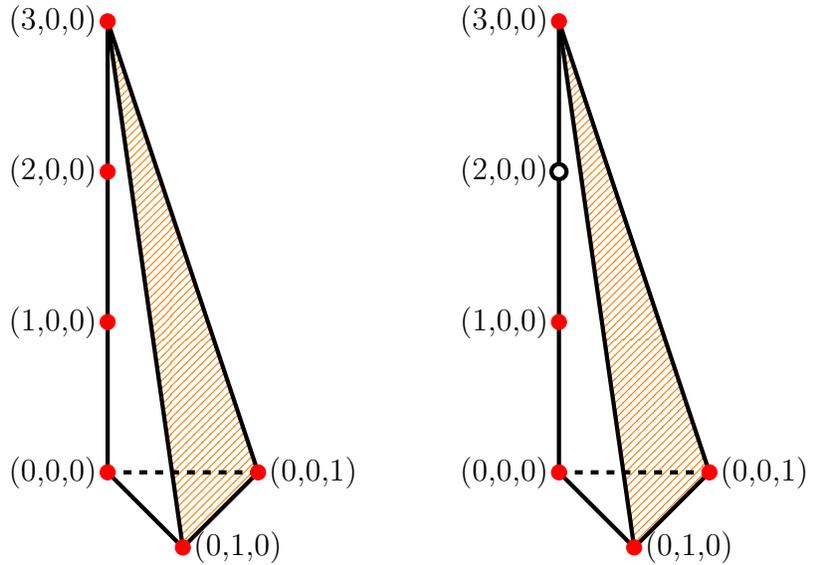
		\node[below] at (2,-0.5) {{\bf Figure 7.} Two similar examples with different forking paths singularities at infinity.};
		\node[left] at (0,1) {(0,0,0)};
		\node[left] at (6,1) {(0,0,0)};
		\node[left] at (0,5) {(2,0,0)};
		\node[left] at (0,7) {(3,0,0)};
		\node[left] at (0,3) {(1,0,0)};
		\node[left] at (6,5) {(2,0,0)};
		\node[left] at (6,7) {(3,0,0)};
		\node[left] at (6,3) {(1,0,0)};
		\node[right] at (2,1) {(0,0,1)};
		\node[right] at (1,0) {(0,1,0)};
		\node[right] at (7,0) {(0,1,0)};
		\node[right] at (8,1) {(0,0,1)};
		\end{tikzpicture}
	\end{center}
	We have that $\conv(A_1)=\conv(A_2)=\Delta,$ and $\vol(\Delta)=3.$ In this case, the polygon $P$ is a standard simplex of size $3,$ therefore, $\area(P)=9.$ In both cases, the horizontal facet contributes exactly $1$ punctured point. Also, in both cases, we have exactly one facet $\Gamma\subset\Delta$ (hatched orange) such that $\ind_v(A_1\cap\Gamma)\neq 1$ and $\ind_v(A_2\cap\Gamma)\neq 1.$ The area of this facet is equal to $1.$
	
	In the first case, we have $i^{\Gamma}=(3,1,\ldots),$ while in the second case $i^{\Gamma}=(3,3,1,\ldots).$ 
	Substituting all of this data into (\ref{mainformula}), we get:
	\begin{align*}
	&\mathcal{D}_1=\dfrac{9-2\cdot 3+1-1\cdot(3-1)}{2}=\dfrac{9-6+1-2}{2}=1;\\
	&\mathcal{D}_2=\dfrac{9-2\cdot 3+1-1\cdot((3-1)+(3-1))}{2}=\dfrac{9-6+1-4}{2}=0.
	\end{align*}
\end{exa}

\subsection{The Tropical Counterpart}

It is very natural to study the behaviour of singularities of generic objects under passing to the tropical limit. There is a number of papers in which this problem is addressed, we first give a short overview. 

According to Mikhalkin's correspondence theorem (see \cite{Mikh} and \cite{S}), under passing to the tropical limit, the nodes of a generic nodal curve in $\CC^2$ are accumulated with certain multiplicities in triple points and crossings of the corresponding tropical curve, as well as the midpoints of its multiple edges. 

The main result of \cite{MMS1} is the classification of singular tropical curves of maximal-dimensional geometric type. A singular point of such a tropical curve of this type is either a crossing of two edges, or a multiplicity $3$ three-valent vertex, or a point on a midpoint of an edge of weight $2.$ The work \cite{MMS2} by the same authors addresses a similar problem for singular tropical surfaces in the $3$-dimensional space.

The work \cite{MR} is devoted to the study of behaviour of log-inflection points of curves in $\CC^2$ under the passage to the tropical limit. Assuming that the limiting tropical curve is smooth, the authors show that the log-inflection points accumulate by pairs at the midpoints of its bounded edges.

In \cite{LM}, lifts of tropical bitangents (i.e., nodes of the dual curve) to the tropicalization of a given complex algebraic curve and their lifting multiplicities are studied. In particular, it is shown that all seven bitangents of a smooth tropical plane quartic lift in sets of four to algebraic bitangents.

Algebraic and combinatorial aspects of projections of $m-$dimensional tropical varieties onto $(m+1)-$dimensional planes are studied in \cite{HT}. For the case of curves (which is exactly the tropical version of the problem addressed in our paper), some bounds for the number of self-intersection of projections onto the plane were given, as well as constructions with many self-intersections.   

Let $b\colon A\to\Z$ be a sufficiently generic function. Consider the family $\tilde{\mathcal{C}}_{\tau}=\{f_{1,\tau}=\ldots=f_{n+1,\tau}=0\},$ where $$f_{j,\tau}=\sum\limits_{\alpha\in A}c_{j,\alpha}(\tau)x_1^{\alpha_1}\ldots x_n^{\alpha_n}y^{\alpha_{n+1}}t^{\alpha_{n+2}},$$ and every coefficient $c_{j,\alpha}(\tau)$ is a generic polynomial in $\tau$ of degree $b(\alpha).$ Denote by $\mathcal C_{\tau}$ the closure of the image $\pi(\tilde{\mathcal{C}}_{\tau}).$ 

The tropical limit of $\tilde{\mathcal{C}}_{\tau}$ as $\tau\rightarrow\infty$ is a tropical curve $\tilde{C}$. By $C$ denote its image under the projection forgetting the first $n$ coordinates.
\begin{prb}
In the same notation as above, which points of the tropical curve $C$ are the tropical limits of nodes of the curves $\mathcal C_{\tau}$ as $\tau\rightarrow\infty$? How many nodes does each of those points ``accumulate''? 
\end{prb}

The example below illustrates the case when the tropical limit of the nodes of the curves $\mathcal{C}_{\tau}$ is a node of the corresponding plane tropical curve $C.$ 
\begin{exa}\label{trop1}
Take $A=\{(0,0,0),(1,0,0),(2,0,0),(1,1,0),(0,0,1)\}.$ The volume of the polytope $\Delta=\conv(A)$ is equal to $2.$ Its fiber polytope with respect to the projection $\pi\colon\CC^3\to\CC^2$ forgetting the first coordinate is of area $6$. For each of the facets $\Gamma\in\mathcal{F}(\Delta)$ we have $\ind_v(A\cap\Delta)=1,$ and the set $\mathcal H(\Delta)$ of the horizontal facets is empty. Therefore, by Theorem \ref{developed}, the number of nodes of the projection of a generic complete intersection defined by polynomials $f,g\in\C^A,$ is equal to $\dfrac{6-2\cdot 2}{2}=1.$ 

\begin{center}
\begin{tikzpicture}[scale=0.9]
\draw[dashed, ultra thick] (0,0)--(2,2);
\draw[dashed, ultra thick] (0,0)--(0,4);
\draw[dashed, ultra thick] (0,0)--(-1,-2);
\draw[ultra thick] (0,4)--(-1,-2);
\draw[ultra thick] (0,4)--(2,2);
\draw[ultra thick] (2,2)--(-1,-2);
\draw[fill,red] (0,0) circle [radius=0.1];
\draw[fill,red] (2,2) circle [radius=0.1];
\draw[fill,red] (-1,-2) circle [radius=0.1];
\draw[fill,red] (0,2) circle [radius=0.1];
\draw[fill,red] (0,4) circle [radius=0.1];
\draw[ultra thick] (6,0)--(10,0);
\draw[ultra thick] (6,0)--(6,4);
\draw[ultra thick] (6,4)--(10,2);
\draw[ultra thick] (10,0)--(10,2);
\draw[fill] (6,0) circle [radius=0.1];
\draw[fill] (10,0) circle [radius=0.1];
\draw[fill] (10,2) circle [radius=0.1];
\draw[fill] (6,4) circle [radius=0.1];
\draw[fill] (6,2) circle [radius=0.1];
\draw[fill] (8,0) circle [radius=0.1];
\draw[fill] (8,2) circle [radius=0.1];
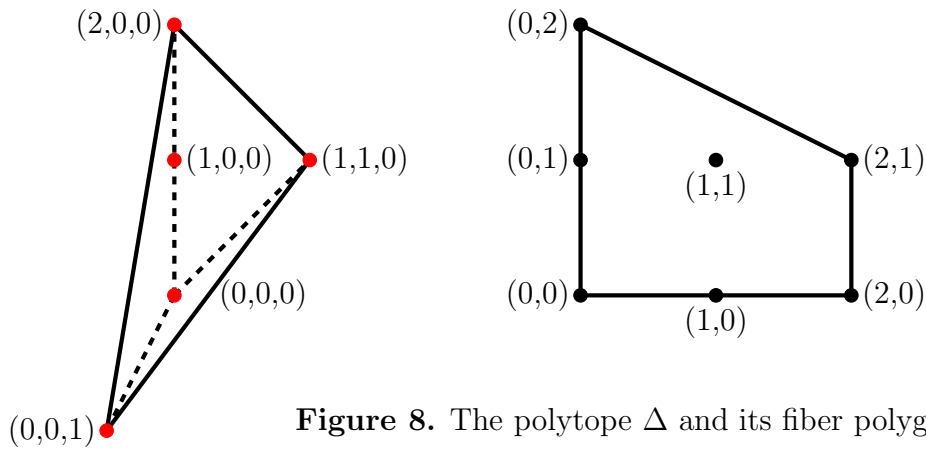
\node[below] at (7,-1.5) {{\bf Figure 8.} The polytope $\Delta$ and its fiber polygon $P.$};
\node[right] at (0.5,0) {(0,0,0)};
\node[left] at (-1,-2) {(0,0,1)};
\node[right] at (2,2) {(1,1,0)};
\node[right] at (0,2) {(1,0,0)};
\node[left] at (0,4) {(2,0,0)};
\node[left] at (6,0) {(0,0)};
\node[left] at (6,4) {(0,2)};
\node[left] at (6,2) {(0,1)};
\node[right] at (10,2) {(2,1)};
\node[right] at (10,0) {(2,0)};
\node[below] at (8,2) {(1,1)};
\node[below] at (8,0) {(1,0)};
\end{tikzpicture}
\end{center}

One can easily express the coordinates $(u,v)$ of this unique double point in terms of the coefficients of the polynomials $f,g.$ Indeed, suppose that we have $f=a_0+a_1 x+a_2 x^2+a_3 t+a_4 x y$ and $g=b_0+b_1 x+b_2 x^2+b_3 t+b_4 x y.$ If $f(x,u,v)$ and $g(x,u,v)$ have two common roots as polynomials in the variable $x,$ then, the numbers $u,v$ clearly should satisfy the following equations: 
	\begin{equation}\label{coord_node}
\left\{\begin{aligned}
(a_1+a_4u)b_2=(b_1+b_4u)a_2\\
(a_0+a_3v)b_2=(b_0+b_3v)a_2\\
 \end{aligned}\right. \iff
 \left\{\begin{aligned}
u=\dfrac{a_1b_2-b_1a_2}{a_2b_4-b_2a_4}\\
v=\dfrac{a_0b_2-b_0a_2}{a_2b_3-b_2a_3}
\end{aligned}\right. 
	.\end{equation}
	
Now, consider the following deformation of the polynomials $f$ and $g:$
\begin{align}\label{deform}
\tilde{f}_{\tau}(x,y,t)&=\tau^{-1}+\tau^{-2}x+\tau^{-8}x^2+\tau^{-4}t+\tau^{-8}xy;\\
\tilde{g}_{\tau}(x,y,t)&=\tau^{5}+\tau^{5}x+\tau^{5}x^2+\tau^{5}t+\tau^{6}xy;
\end{align}
Passing to the tropical limit, we obtain a pair of tropical polynomials $F(X,Y,T)$ and $G(X,Y,T),$ where: 
\begin{align*}
F(X,Y,T)&=\max(-1,X-2,2X-8,T-4,X+Y-8)\\
G(X,Y,T)&=\max(5,X+5,2X+5,T+5,X+Y+6)\\
\end{align*}
The intersection of the tropical hypersurfaces defined by polynomials $F$ and $G$ is a $3$-valent tropical curve $\tilde{C}$. Its image $C$ under the projection forgetting the first coordinate and the corresponding subdivision of the polygon $P$ are shown in Figure $9$ below.

\begin{center}
\begin{tikzpicture}[scale=0.7]
\draw[ultra thick](1,4)--(5,8);
\draw[ultra thick](1,3)--(1,4);
\draw[ultra thick] (1,3)--(8,3);
\draw[ultra thick] (1,3)--(0,2);
\draw[ultra thick] (0,2)--(-5,2);
\draw[ultra thick] (0,2)--(0,-3);
\draw[ultra thick] (1,4)--(-5,4);
\draw[ultra thick] (5,8)--(7,12);
\draw[ultra thick] (5,8)--(5,-3);
\draw[fill,blue] (1,4) circle [radius=0.2];
\draw[fill,blue] (5,8) circle [radius=0.2];
\draw[fill,red] (5,3) circle [radius=0.2];
\draw[fill,blue] (0,2) circle [radius=0.2];
\draw[fill,blue] (1,3) circle [radius=0.2];
\node[above left] at (1.3,4) {$(1,4)$};
\node[below right] at (5,8) {$(5,8)$};
\node[above right] at (5,3) {$(5,3)$};
\node[below left] at (0,2) {$(0,2)$};
\node[below right] at (0.8,3) {$(1,3)$};
\node[above left] at (2,6) {$2T+2$};
\node[above right] at (6,5) {$2Y+T$};
\node[above] at (3.2,3.5) {$Y+T+5$};
\node[right] at (6,1) {$2Y+3$};
\node[left] at (0,3) {$T+6$};
\node[below left] at (-2,0) {$8$};
\node[below] at (3,1) {$Y+8$};
\draw[ultra thick] (12,-2)--(16,-2);
\draw[ultra thick] (12,-2)--(12,2);
\draw[ultra thick] (12,2)--(16,0);
\draw[ultra thick] (16,-2)--(16,0);
\draw[ultra thick] (14,0)--(14,-2);
\draw[ultra thick] (12,0)--(14,-2);
\draw[ultra thick] (12,0)--(16,0);
\draw[ultra thick] (12,2)--(14,0);
\draw[fill] (12,-2) circle [radius=0.1];
\draw[fill] (16,-2) circle [radius=0.1];
\draw[fill] (14,-2) circle [radius=0.1];
\draw[fill] (12,0) circle [radius=0.1];
\draw[fill] (14,0) circle [radius=0.1];
\draw[fill] (16,0) circle [radius=0.1];
\draw[fill] (12,2) circle [radius=0.1];
\node[left] at (12,-2) {$(0,0)$};
\node[left] at (12,2) {$(0,2)$};
\node[left] at (12,0) {$(0,1)$};
\node[right] at (16,0) {$(2,1)$};
\node[right] at (16,-2) {$(2,0)$};
\node[below right] at (14,0) {$(1,1)$};
\node[below] at (14,-2) {$(1,0)$};
\node[below] at (7,-5) {{\bf Figure 9.} The tropical curve $C$ and the corresponding subdivision of the polygon $P.$};
\end{tikzpicture}
\end{center}
The tropical curve $C$ has one node (marked red). Now, let us substitute the coefficients of the polynomials $\tilde{f}_{\tau}(x,y,t)$ and $\tilde{g}_{\tau}(x,y,t)$ into the expressions for the coordinates $(u,v)$ of the node of the curve $\mathcal C$ (see (\ref{coord_node})). We get: 
\begin{equation}
\left\{\begin{aligned}
u(\tau)=\dfrac{a_1(\tau)b_2(\tau)-b_1(\tau)a_2(\tau)}{a_2(\tau)b_4(\tau)-b_2(\tau)a_4(\tau)}\\
v(\tau)=\dfrac{a_0(\tau)b_2(\tau)-b_0(\tau)a_2(\tau)}{a_2(\tau)b_3(\tau)-b_2(\tau)a_3(\tau)}
\end{aligned}\right. \Longrightarrow
\left\{\begin{aligned}
u(\tau)&=\dfrac{-\tau^{-3}+\tau^3}{-\tau^{-3}+\tau^{-2}}=1+\tau+\tau^2+\tau^3+\tau^4+\tau^5\\
v(\tau)&=\dfrac{-\tau^{-3}-\tau^4}{\tau^{-3}-\tau}=\dfrac{\tau^7-1}{1-\tau^4}\end{aligned}\right. 
\end{equation}
Note that as $\tau\rightarrow\infty,$ we have  $u(\tau)\sim\tau^5$ and $v(\tau)\sim-\tau^3,$ which agrees with the coordinates of the node of the tropical curve $C$ being $(5,3).$
\end{exa}

\begin{exa}\label{trop2}
Now, consider the following deformation of the polynomials $f$ and $g:$
\begin{align}\label{deform2}
\tilde{f}_{\tau}(x,y,t)&=\tau+\tau^{4}x+\tau x^2+\tau t+\tau^{5}xy;\\
\tilde{g}_{\tau}(x,y,t)&=\tau+\tau x+\tau x^2+\tau^{2}t+\tau^{2}xy;
\end{align}
Passing to the tropical limit, we obtain a pair of tropical polynomials $F(X,Y,T)$ and $G(X,Y,T),$ where: 
\begin{align*}
F(X,Y,T)&=\max(1,X+4,2X+1,T+1,X+Y+5)\\
G(X,Y,T)&=\max(1,X+1,2X+1,T+2,X+Y+2)\\
\end{align*}
The intersection of the tropical hypersurfaces defined by polynomials $F$ and $G$ is a $3$-valent tropical curve $\tilde{C}$. Its image $C$ under the projection forgetting the first coordinate and the corresponding subdivision of the polygon $P$ are shown in Figure $10$ below.

\begin{center}
	\begin{tikzpicture}[scale=0.7]
	\draw[ultra thick](-1,-4)--(-1,5);
	\draw[ultra thick](-5,-1)--(4,-1);
	\draw[ultra thick] (-5,5)--(-1,5);
	\draw[ultra thick] (-1,5)--(2,11);
	\draw[fill,blue] (-1,-1) circle [radius=0.2];
	\draw[fill,blue] (-1,5) circle [radius=0.2];
	\draw[fill,red] (-1,0) circle [radius=0.2];
    \node[below left] at (-1,-1) {$(-1,-1)$};
	\node[below right] at (-1,5) {$(-1,5)$};
	\node[right] at (-1,0) {$(-1,0)$};
	\node[above left] at (-1,8) {$2T+6$};
	\node[above] at (2,1.5) {$2Y+T+13$};
    \node[left] at (-2,2) {$T+11$};
	\node[below left] at (-3,-3) {$10$};
	\node[below] at (2,-2) {$Y+11$};
	\draw[ultra thick] (8,-2)--(12,-2);
	\draw[ultra thick] (8,-2)--(8,2);
	\draw[ultra thick] (8,2)--(12,0);
	\draw[ultra thick] (12,-2)--(12,0);
    \draw[ultra thick] (8,0)--(12,0);
	\draw[fill] (8,-2) circle [radius=0.1];
	\draw[fill] (12,-2) circle [radius=0.1];
	\draw[fill] (10,-2) circle [radius=0.1];
	\draw[fill] (8,0) circle [radius=0.1];
	\draw[fill] (10,0) circle [radius=0.1];
	\draw[fill] (12,0) circle [radius=0.1];
	\draw[fill] (8,2) circle [radius=0.1];
	\node[left] at (8,-2) {$(0,0)$};
	\node[left] at (8,2) {$(0,2)$};
	\node[left] at (8,0) {$(0,1)$};
	\node[right] at (12,0) {$(2,1)$};
	\node[right] at (12,-2) {$(2,0)$};
	\node[below right] at (10,0) {$(1,1)$};
	\node[below] at (10,-2) {$(1,0)$};
	\node[below] at (3,-5) {{\bf Figure 10.} The tropical curve $C$ and the corresponding subdivision of the polygon $P.$};
	\end{tikzpicture}
\end{center}
Now, let us substitute the coefficients of the polynomials $\tilde{f}_{\tau}(x,y,t)$ and $\tilde{g}_{\tau}(x,y,t)$ into the expressions for the coordinates $(u,v)$ of the node of the curve $\mathcal C$ (see (\ref{coord_node})). We get: 
\begin{equation}
\left\{\begin{aligned}
u(\tau)=\dfrac{a_1(\tau)b_2(\tau)-b_1(\tau)a_2(\tau)}{a_2(\tau)b_4(\tau)-b_2(\tau)a_4(\tau)}\\
v(\tau)=\dfrac{a_0(\tau)b_2(\tau)-b_0(\tau)a_2(\tau)}{a_2(\tau)b_3(\tau)-b_2(\tau)a_3(\tau)}
\end{aligned}\right. \Longrightarrow
\left\{\begin{aligned}
u(\tau)&=\dfrac{\tau^{5}-\tau^2}{\tau^{3}-\tau^{6}}=-\dfrac{\tau^2(\tau^3-1)}{\tau^3(\tau^3-1)}\\
v(\tau)&=\dfrac{\tau^{2}-\tau^2}{\tau^{3}-\tau^2}=0\end{aligned}\right. 
\end{equation}
So the tropical limit of the nodes of the curves $\mathcal{C}_{\tau}$ in this case is the point $(-1,0)$  (marked red). Unlike the case considered in Example \ref{trop1}, this point is not a vertex of the corresponding nodal tropical curve. Moreover, the edge containing this point is of multiplicity 2. 
\end{exa}

{\bf Arina Voorhaar (Arkhipova)}: National Research University Higher School of Economics, Russian Federation, Department of Mathematics, 6 Usacheva st, Moscow 119048;

University of Geneva, Switzerland, Department of Mathematics, Route de Drize 7, 1227 Carouge GE

\href{mailto:aarhipova@hse.ru}{\nolinkurl{aarhipova@hse.ru}}\\ 
\href{mailto:Arina.Arkhipova@unige.ch}{\nolinkurl{Arina.Arkhipova@unige.ch}}
\end{document}